\documentclass[a4paper,reqno]{amsart}

\usepackage{amsmath,amssymb,amsthm}
\usepackage{graphicx}
\usepackage{cite}
\usepackage{xcolor}
\usepackage{mathtools}
\usepackage{enumerate}
\usepackage{nicefrac}
\usepackage{bm}
\usepackage{a4wide}
\usepackage{algorithm,algorithmicx,algpseudocode}
\usepackage{booktabs}
\usepackage{refcheck}

\def\M{\mathcal M}

\newcommand{\dif}{\,\text{d}}

\renewcommand{\phi}{\varphi}

\newcommand{\jump}[1]{\lbrack\!\lbrack#1\rbrack\!\rbrack} 
\makeatletter
\newcommand*\bigcdot{\mathpalette\bigcdot@{.5}}
\newcommand*\bigcdot@[2]{\mathbin{\vcenter{\hbox{\scalebox{#2}{$\m@th#1\bullet$}}}}}
\makeatother
\newcommand{\tightoverset}[2]{%
  \mathop{#2}\limits^{\vbox to -.45ex{\kern-0.75ex\hbox{$#1$}\vss}}}
\makeatletter
\DeclareRobustCommand\widecheck[1]{{\mathpalette\@widecheck{#1}}}
\def\@widecheck#1#2{%
    \setbox\z@\hbox{\m@th$#1#2$}%
    \setbox\tw@\hbox{\m@th$#1%
       \widehat{%
          \vrule\@width\z@\@height\ht\z@
          \vrule\@height\z@\@width\wd\z@}$}%
    \dp\tw@-\ht\z@
    \@tempdima\ht\z@ \advance\@tempdima2\ht\tw@ \divide\@tempdima\thr@@
    \setbox\tw@\hbox{%
       \raise\@tempdima\hbox{\scalebox{1}[-1]{\lower\@tempdima\box
\tw@}}}%
    {\ooalign{\box\tw@ \cr \box\z@}}}
\makeatother

\newtheorem{Remark}[equation]{Remark}
\newenvironment{remark}{\begin{Remark}\rm}{\end{Remark}}
\newtheorem{theorem}[equation]{Theorem}

\newtheorem{corollary}[equation]{Corollary}
\newtheorem{lemma}[equation]{Lemma}

\numberwithin{equation}{section}

\title{Pointwise a posteriori error bounds for blow-up in the semilinear heat equation}
\author{Irene Kyza}
\address{Department of Mathematics, University of Dundee, Dundee, DD1 4HN, United Kingdom}
\email{ikyza@maths.dundee.ac.uk}

\author{Stephen Metcalfe}
\address{Department of Chemistry, University of York, York, YO10 5DD, United Kingdom}
\email{smetcalfephd@gmail.com}

\thanks{S. Metcalfe acknowledges the support of the Swiss National Science Foundation (SNF) grant no. 200021-162990}

\begin{document}

\begin{abstract}
This work is concerned with the development of a space-time adaptive numerical method, based on a rigorous \emph{a posteriori} error bound, for the semilinear heat equation with a general local Lipschitz reaction term whose solution may blow-up in finite time. More specifically, conditional \emph{a posteriori} error bounds are derived in the $L^{\infty}L^{\infty}$ norm for a first order in time, implicit-explicit (IMEX), conforming finite element method in space discretization of the problem. Numerical experiments applied to both blow-up and non blow-up cases highlight the generality of our approach and complement the theoretical results.
\end{abstract}

\keywords{Semilinear heat equation, IMEX methods, conditional \emph{a posteriori} error estimates, blow-up singularities}

\subjclass[2010]{65J08, 65L05, 65L60}

\maketitle

\section{Introduction}

Let $\displaystyle \Omega \subset \mathbb{R}^d$ with $d=2$ or $d=3$ be a bounded polyhedral domain and consider the problem
\begin{equation}\label{model_strong}
\begin{aligned}
u_t - a\Delta{u}-f(u) &= 0 \qquad && \text{in }  \Omega, \mbox{ } t > 0  \mbox{,} \\ u &=0 \mbox{ } && \text{on }  \partial\Omega, \mbox{ } t > 0 \mbox{,} \\ u(\cdot,0) &=u_0 \mbox{ } && \text{in } \bar{\Omega}\mbox{,}
\end{aligned}
\end{equation}
where $a$ is a positive constant and the initial condition $u_0 \in C^2(\Omega)\cap C(\bar{\Omega})$ takes boundary values that are compatible with those of the PDE. Note that here the reaction term $f$ can be both space and time dependent but as the nature of the dependence is usually clear, we omit writing it explicitly for brevity. It is well known that for certain data the solution to \eqref{model_strong} exhibits \emph{finite time blow-up} \cite{Hu}, that is, there exists a maximal time of existence $T_{\infty} < \infty$ referred to as the \emph{blow-up time} such that \eqref{model_strong} holds and 
\[
\|u(t)\|_{L^{\infty}(\Omega)}<\infty\text{ for }0<t<T_\infty,\qquad \qquad \lim_{t\nearrow T_\infty}\|u(t)\|_{L^{\infty}(\Omega)} = \infty.
\]
If the solution to \eqref{model_strong} does not exhibit finite-time blow-up then the solution is global and so $T_{\infty} = \infty$. Either way, we assume that \eqref{model_strong} holds on some closed interval $[0,T]$ and that $T < T_{\infty}$. We will see in the sequel that we can show that \eqref{model_strong} has a unique local solution $u \in C(0,T;L^{\infty}(\Omega))$ provided that an implicit \emph{local a posteriori criteron} is satisfied and that this local critereon is robust with respect to the distance from the blow-up time.

The numerical approximation of blow-up phenomena in partial differential equations (PDEs) is a challenging problem due to the high spatial and temporal resolution needed close to the blow-up time. Classical numerical methods that give good approximations to the solution of \eqref{model_strong} close to the blow-up time include the rescaling algorithm of Berger and Kohn \cite{BK88,NZ16} and the MMPDE method \cite{BHR96,HMR08}. Recently, there has been a lot of interest in deriving \emph{a posteriori} error bounds for such problems and using the resulting estimators to drive an adaptive procedure in order to get close to the blow-up time. Indeed, it is easy to see why such an approach can confer significant advantages; for example, if it is known that the $L^{\infty}L^{\infty}$ norm of the error is always bounded from above by a finite quantity then it is impossible to surpass the blow-up time! 

\emph{A posteriori} error estimators for linear problems tend to be \emph{unconditional}, that is, they always hold independent of the problem data and the size of the discretization parameters. For nonlinear problems, the situation is more complicated since the existence of a solution to an appropriate error equation (and, thus, of an error bound) usually requires that either the data or the discretization parameters are sufficiently small. As a result, \emph{a posteriori} error estimators for nonlinear problems tend to be \emph{conditional}, that is, they only hold provided that an \emph{a posteriori} verifiable condition (which can be either explicit or implicit) is satisfied. For nonlinear time-dependent problems, there are two commonly used approaches for deriving conditional \emph{a posteriori} error bounds: continuation arguments, cf. \cite{B05, M15,KMW16,CGKM16, KNS04,GM14,CGMS17}, and fixed point arguments, cf. \cite{KM11, CM08, K09, MTKO17a, MTKO17b, MTKO17c}.

The derivation of such estimates for \eqref{model_strong} and related problems in the context of blow-up was first explored in \cite{K09,KM11} for polynomial nonlinearities but these early pointwise bounds are not well suited for the practical computation of blow-up problems by virtue of being \emph{global} rather than \emph{local} in nature. The situation was improved in \cite{M15,CGKM16} by the derivation of error bounds using energy techniques combined with the Gagliardo-Nirenberg inequality that are valid under local rather than global conditions. While the bounds of \cite{M15,CGKM16} represent a significant improvement from a practical perspective, the derived error bounds still have significant drawbacks; for example, the range of nonlinearities that can be considered is smaller than in \cite{K09,KM11} due to Sobolev embedding restrictions and convergence towards the blow-up time is still \emph{slow} when compared with results on the numerical approximation of blow-up in ODEs \cite{KMW16,M15,CGKM16}. It should be remarked, though, that the use of energy techniques in \cite{M15,CGKM16} does confer advantages in other areas; specifically, it allows for the derivation of error bounds for problems with non-symmetric spatial operators for which pointwise error bounds are unlikely to be achievable anytime soon.

In this paper, we seek to derive conditional \emph{a posteriori} error bounds in the $L^\infty L^\infty$ norm for the first order in time, implicit-explicit (IMEX), conforming finite element method in space discretization of \eqref{model_strong}. It is worth noting that the choice of an IMEX method not only confers advantages in terms of ease of solubility but has also been shown to have advantages in the context of estimation of the blow-up time, cf. \cite{M15,CGKM16}. The results that we will present here are an improvement over existing results in the literature in two major ways. Firstly, we significantly broaden the range of nonlinearities under consideration; specifically, the (possibly) nonlinear reaction term $f:\bar{\Omega} \times [0,T_\infty)\times \mathbb{R}\to \mathbb{R}$ is assumed to be continuous and to satisfy the \emph{local} Lipschitz estimate
\begin{equation}\label{eq:Lip}
|f(x,t,v)-f(x,t,w)| \leq  \mathcal{L}(t,|v|,|w|)|v-w| \qquad \forall x \in \bar{\Omega} \quad \forall t \in [0,T_\infty) \quad \forall v,w \in \mathbb{R}.
\end{equation}
Here, $\mathcal{L}:[0,T_\infty)\times \mathbb{R}^+_0 \times \mathbb{R}^{+}_0 \to  \mathbb{R}^{+}_0$ is a \emph{known} function that satisfies $\mathcal{L}(\cdot,a,b) \in L^1(0,T_\infty)$ for any $a,b \in \mathbb{R}^+_0$ and that is continuous and monotone increasing in the second and third arguments. This condition on $f$ is quite general and includes many nonlinearities of interest, for example, it covers any polynomial nonlinearity with suitably regular coefficients as well as nonlinearities of exponential type \cite{KMW16}. We stress that this local Lipschitz assumption is in contrast to assumptions made for currently existing pointwise \emph{a posteriori} bounds of \eqref{model_strong}, cf. \cite{KL16,KL17}, wherein the focus is on the singularly preturbed case and, thus, $f$ is assumed to be globally Lipschitz. Secondly, we follow the approach taken in \cite{KL16,K09,KM11,KL17} of conducting the error analysis via semigroup techniques {\bf --} this allows us to consider the error in an ODE setting. In combination with a \emph{local-in-time} fixed point argument along the lines of \cite{KMW16,CGKM16,M15}, this restores the optimality that is otherwise lost in an energy setting \cite{CGKM16}. Finally, we show numerically that our conditional \emph{a posteriori} error bound is well-behaved with respect to the distance from the blow-up time; specifically, we show that the rate of convergence to the blow-up time is comparable to the rate observed in \cite{KMW16}.

\subsection*{Outline} We begin in Section 2 by outlining the IMEX discretization of \eqref{model_strong} then in Section 3 we introduce several auxiliary results which will be used in the error analysis. In Section 4, we derive the conditional \emph{a posteriori} error bound and in Section 5 we propose a general adaptive algorithm, applicable to both blow-up and fxed-time problems, that is based upon the derived \emph{a posteriori} error bound. We apply this adaptive algorithm to several numerical examples in Section 6 in order to illustrate that the proposed \emph{a posteriori} bound is well-behaved close to the blow-up time. Finally, we draw conclusions and outline our plans for future research in Section 7.

\subsection*{Notation} 

As they will be used frequently in this work, we denote the $L^2$ inner product on $\Omega$ by $(\cdot,\cdot)$ and the $L^{\infty}$ norm on $\Omega$ by $||\cdot||$.

\section{Discretization} 

Consider a shape-regular mesh $\mathcal{T}=\{K\}$ of $\Omega$ with $K$ denoting a generic element of diameter $h_K$ that is constructed via affine mappings $F_{K}:\widehat{K}\to K$ with non-singular Jacobian where $\widehat{K}$ is the $d$-dimensional reference simplex or the $d$-dimensional reference cube. The mesh is allowed to contain a uniformly fixed number of regular hanging nodes per face. With these definitions, the finite element space $\mathcal{V}_h(\mathcal{T})$ over the mesh $\mathcal{T}$ is given by
\begin{equation}\label{eq:FEspace}
\mathcal{V}_h(\mathcal{T}) := \{v \in  H^1_0(\Omega) : v|_K\circ F_K \in \mathcal{P}^p(\widehat{K}), \, K \in \mathcal{T} \},
\end{equation}
where $\mathcal{P}^p(\widehat{K})$ denotes the space of polynomials of total degree $p$ if $\widehat{K}$ is the $d$-dimensional reference simplex or of degree $p$ in each variable if $\widehat{K}$ is the $d$-dimensional  reference cube. Given two meshes $\mathcal{T}_1$ and $\mathcal{T}_2$, we denote their \emph{coarsest common refinement} by $\mathcal{T}_1 \vee \mathcal{T}_2$ and their \emph{finest common coarsening} by $\mathcal{T}_1 \wedge \mathcal{T}_2$. We also define the jump residual $\jump{\nabla v_h}$ of a function $v_h \in \mathcal{V}_h(\mathcal{T})$ at a point $x$ on the $(d-1)$-dimensional inter-element face $E = \bar{K} \cap \bar{K}'$, $K, K' \in \mathcal{T}$ by 
\begin{equation*}
\jump{\nabla v_h}(x) := \lim_{\delta \rightarrow 0} [\nabla v_h (x + \delta n) - \nabla v_h (x - \delta n)] \bigcdot n,
\end{equation*}
where $n$ is an arbitrary normal vector on $E$.

We consider a first order in time, implicit-explicit (IMEX), space-time discretization of \eqref{model_strong} consisting of implicit treatment for the diffusion term and explicit treatment for the nonlinear reaction term. For the spatial discretization, we use the standard conforming finite element method. To this end, we introduce a sequence of time nodes $0 := t_0 < t_1 < \cdots < t_{M-1} < t_M := T$ which define a time partition $\M:=\{I_m\}_{m=1}^M$ of~$(0,T)$ into~$M$ open time intervals~$I_m:=(t_{m-1},t_m)$, $m=1,\ldots,M$. The length $k_m := t_m - t_{m-1}$ (which may be variable) of the time interval $I_m$ is called the time step length. Furthermore, if we let $\mathcal{T}_0$ denote an initial spatial mesh of $\Omega$ associated with the first time node $t_0=0$ then to each additional time node $t_m$, $m=1,\dots,M$, we associate a spatial mesh $\mathcal{T}_m$ which is assumed to have been obtained from $\mathcal{T}_{m-1}$ by local refinement and/or coarsening. We remark that this restriction upon mesh change is made in order to avoid degradation of the finite element solution, cf. \cite{BKM13,D82}. To each mesh $\mathcal{T}_{m}$ we then assign the finite element space $\mathcal{V}_h^m := \mathcal{V}_h(\mathcal{T}_m)$ given by~\eqref{eq:FEspace}. 

With this notation at hand, the IMEX method then reads as follows. Following the approach taken in \cite{BKM13}, we choose $U^0 \in \mathcal{V}_h^0$ to be the unique solution of the problem
\begin{equation}
\begin{aligned}
\label{IC}
(\nabla U^{0}, \nabla V^0) = (-\Delta u_0,V^0)  \qquad \forall V^0 \in \mathcal{V}_h^0.
\end{aligned}
\end{equation}
We then seek $U^m \in \mathcal{V}_h^m$, $m = 1, \ldots, M$, such that 
\begin{equation}
\begin{aligned}
\label{IMEX}
\bigg (\frac{U^{m}-U^{m-1}}{k_m},V^m \bigg)+a(\nabla U^{m}, \nabla V^m) - (f^{m-1},V^m) = 0 \qquad \forall V^m \in \mathcal{V}_h^m,
\end{aligned}
\end{equation}
where we set $f^{m-1}:=f(\cdot,t_{m-1},U^{m-1})$ for brevity.  For $t \in \bar{I}_m$, $m = 1, \ldots, M$, $U(t)$ is then defined to be the linear interpolant with respect to $t$ of the values $U^{m-1}$ and $U^{m}$, viz., 
\begin{equation}
\begin{aligned}
U(t) := \ell_{m-1}(t) U^{m-1} + \ell_{m}(t) U^m,
\end{aligned}
\end{equation}
where $\{\ell_{m-1}, \ell_{m}\}$ denotes the standard linear Lagrange interpolation basis on the interval $\bar{I}_m$.

\section{Preliminaries}

Before we proceed with the error analysis, we require some auxiliary results and some additional notation. Our first result is a maximum principle for a related parabolic equation.

\begin{theorem}
\label{maxprinciple}
Let $\mathrm{e}^{t \Delta}$ be the solution operator for the problem
\begin{equation}
\begin{aligned}
\notag
w_t - a\Delta{w} &= 0 \qquad && \text{in }  \Omega, \mbox{ }t>0  \mbox{,} \\ w &=0 \mbox{ } && \text{on }  \partial\Omega, \mbox{ }t>0 \mbox{,} \\ w(\cdot,0) &=w_0 \mbox{ } && \text{in } \bar{\Omega}\mbox{,}
\end{aligned}
\end{equation}
with $w_0 \in L^{\infty}(\Omega)$. In other words, $w(t) = \mathrm{e}^{t \Delta}w_0$. Then for any $t>0$, the following bound holds
\begin{equation}
\begin{aligned}
\notag
||\mathrm{e}^{t \Delta}w_0|| \leq ||w_0||.
\end{aligned}
\end{equation}
\end{theorem}
\begin{proof}
See page 93 in \cite{T84}.
\end{proof}

We next introduce an error bound for a related elliptic problem which will be crucial in the error analysis of the parabolic problem.

\begin{theorem}
\label{ellipticerrorbound}
Let $w \in H^1_0(\Omega)\cap C(\bar{\Omega})$ be the unique solution to the elliptic problem
\begin{equation}
\begin{aligned}
\notag
-a\Delta{w} &= g \qquad && \text{in }  \Omega \mbox{,} \\ w &=0 \mbox{ } && \text{on }  \partial\Omega \mbox{,}
\end{aligned}
\end{equation}
where $g \in C({\Omega})$ and let $w_h \in \mathcal{V}_h$ be its conforming finite element approximation. Then the following pointwise a posteriori bound holds
\begin{equation}
\begin{aligned}
\notag
||w-w_h|| \leq C_{\infty}\log(1 \slash \underbar{h}) \max_{K \in \mathcal{T}} \left[ h_K^2 a^{-1}||g+a \Delta w_h||_{L^{\infty}(K)} + h_K||\jump{\nabla w_h}||_{L^{\infty}(\partial K \backslash \partial \Omega)} \right],
\end{aligned}
\end{equation}
where $\displaystyle \underbar{h} := \min_{K \in \mathcal{T}} h_K$ is the minimum mesh-size and $C_{\infty}$ is a positive constant that is independent of the maximum mesh-size, $a$, $w$ and $w_h$ but may be dependent upon the size of the domain $\Omega$.
\end{theorem}
\begin{proof}
See \cite{DK14}.
\end{proof}
The error bound of Theorem \ref{ellipticerrorbound} and, thus, the spatial error estimators in the forthcoming parabolic error analysis are well-behaved in the elliptic regime but badly-behaved in the singularly preturbed regime ($a \approx 0$), cf. \cite{DK14}. As we are primarily interested in blow-up in the elliptic regime, we elect to use the simpler error bound of Theorem \ref{ellipticerrorbound}  but if one is interested in the singularly preturbed regime, the spatial estimators should be replaced with the \emph{full} estimator from \cite{DK14}.

\section{Error Analysis}

In the forthcoming error analysis, on each time step $m$, we will work with the Bochner space $C(\bar{I}_m;L^{\infty}(\Omega))$ {\bf --} the space of all  continuous functions $v:\bar{I}_m \to L^{\infty}(\Omega)$ equipped with the norm
\begin{equation}
\begin{aligned}
\notag
||v||_m := \sup_{t \in \bar{I}_m} ||v(t)||.
\end{aligned}
\end{equation}
In what follows, we will need to make use of the elliptic reconstruction technique \cite{LM06,MN03}. To that end, we define the \emph{elliptic reconstruction} $\omega^m \in H^1_0(\Omega)$, $m = 0,1, \ldots, M$, to be the solution of 
\begin{equation}
\begin{aligned}
\label{recon}
a(\nabla \omega^m, \nabla v)=(A^m,v) \qquad \forall v \in H^1_0(\Omega),
\end{aligned}
\end{equation}
where $A^m$ is the \emph{discrete laplacian} given by 
\begin{equation*}
A^m  := \begin{cases}\displaystyle -a \Delta u_0 \qquad & \text{if } m = 0
 \\ \displaystyle f^{m-1} -  U_t  \big|_{I_m} \qquad & \text{if } m \neq 0 \end{cases}.
\end{equation*}
As with the numerical solution $U(t)$, $\omega(t)$ is defined to be the linear interpolant with respect to $t$ of the values $\omega^{m-1}$ and $\omega^{m}$, viz., 
\begin{equation}
\begin{aligned}
\label{omegadef}
\omega(t) := \ell_{m-1}(t) \omega^{m-1} + \ell_{m}(t) \omega^m,
\end{aligned}
\end{equation}
for $t \in \bar{I}_m$, $m = 1, \ldots, M$. For the error analysis, the error $e := u - U$ will be decomposed $e = \rho + \epsilon$ where $\rho := u - \omega$ is the \emph{parabolic error} and $\epsilon = \omega - U$ is the \emph{elliptic error}. Note that the conforming finite element discretization of \eqref{recon} is either \eqref{IC} or \eqref{IMEX} thus $||\epsilon(t_m)||$ can be estimated through elliptic error estimators available in the literature. To that end, we have the following lemma.
\begin{lemma}
\label{spaceest}
The estimate
\begin{equation}
\begin{aligned}
\notag
||\epsilon(t_m)|| \leq C_{\infty}\log(1 \slash \underbar{h}_m)\! \max_{K \in \mathcal{T}_m} \eta_S^m \big|_K,
\end{aligned}
\end{equation}
holds with $\displaystyle \underbar{h}_m := \min_{K \in \mathcal{T}_m} h_K$ denoting the minimum mesh-size, $C_{\infty}$ the constant of Theorem \ref{ellipticerrorbound} and where $\eta_S^m$ is the primary space estimator given by
\begin{equation}
\begin{aligned}
\notag
\eta_S^m \big|_K:= h_K^2 a^{-1} ||A^m + a\Delta U^m||_{L^{\infty}(K)}+ h_K||\jump{\nabla U^m}||_{L^{\infty}(\partial K \backslash \partial \Omega)}, \qquad K \in \mathcal{T}_m.
\end{aligned}
\end{equation}
\end{lemma}
\begin{proof}
Since $A^m \in C(\Omega)$, this follows directly from Theorem \ref{ellipticerrorbound}.
\end{proof}
To begin construction of the error equation, we first deduce from \eqref{recon} and \eqref{omegadef} that
\begin{equation}
\begin{aligned}
\label{omegaae2}
-a \Delta \omega = \ell_{m-1} A^{m-1} + \ell_{m} A^m,
\end{aligned}
\end{equation}
for any $t \in \bar{I}_m$. We then subtract \eqref{omegaae2} from \eqref{model_strong} to obtain
\begin{equation}
\begin{aligned}
u_t - a\Delta \rho = f(u) -\ell_{m-1} A^{m-1} - \ell_{m} A^m.
\end{aligned}
\end{equation}
Adding and subtracting $f(\omega)$ and $f(U)$ then yields
\begin{equation}
\begin{aligned}
u_t - a\Delta \rho = f(\rho+\omega)-f(\omega)+f(\omega)-f(U)+f(U) -\ell_{m-1} A^{m-1} - \ell_{m} A^m.
\end{aligned}
\end{equation}
Finally, adding and subtracting $\omega_t$ and $U_t$ gives the \emph{error equation}
\begin{equation}
\begin{aligned}
\label{erroreqn}
\rho_t - a\Delta \rho & =  f(\rho+\omega)-f(\omega)+f(\omega)-f(U)+ R_T-\epsilon_t.
\end{aligned}
\end{equation}
where $R_T$ is the \emph{temporal residual} given by $R_T := f(U)-\ell_{m-1} A^{m-1} - \ell_{m} A^m - U_t$.  It can be easily seen that $R_T$ is of optimal order in time by substituting $A^{m-1}$ and $A^m$. Using the temporal residual, we then define the \emph{time estimator} $\eta_T^m$ on each time interval $I_m$ by
\begin{equation}
\begin{aligned}
\notag
\eta_T^m := \int_{I_m} ||R_T(s)|| \dif s.
\end{aligned}
\end{equation}

\subsection{Fixed Point Argument}

We now seek to show that \eqref{erroreqn} has a unique solution $\rho \in \mathcal{B}_m$ where $\mathcal{B}_m$ is the closed ball of radius $\delta_m \psi_m$ centered on zero in the $||\cdot||_m$ norm and where $\delta_m \in [1,\infty)$ is a parameter to be determined with $\psi_m$ chosen such that 
\begin{equation}
\label{psidef1}
||\rho(t_{m-1})|| +  \int_{I_m} ||\epsilon(s)||\mathcal{L}(s,||U(s)||,||U(s)||+||\epsilon(s)||) \dif s+  \eta_T^m + \int_{I_m} ||\epsilon_t|| \dif s  \leq \psi_m.
\end{equation}

To do this, we will use the Banach Fixed Point Theorem to show that $\Phi_m$ given by
\begin{equation}
\begin{aligned}
\notag
\Phi_m(v)(t) & := \mathrm{e}^{(t-t_{m-1}) \Delta} \rho(t_{m-1}) + \int_{t_{m-1}}^t\mathrm{e}^{(t-s) \Delta} [f(v+\omega)-f(\omega)+f(\omega)-f(U)] \dif s  \\
                    & + \int_{t_{m-1}}^t\mathrm{e}^{(t-s) \Delta} [R_T-\epsilon_t ] \dif s,
\end{aligned}
\end{equation}
$t \in \bar{I}_m$, has a unique fixed point $\rho \in \mathcal{B}_m$ which by Duhamel's Principle must also solve \eqref{erroreqn}. To satisfy the criterea of the Banach Fixed Point Theorem, we must show two things:
\smallskip

(1) That $\Phi_m$ maps $\mathcal{B}_m$ onto itself.
\smallskip

(2) That $\Phi_m$ is a contraction mapping.

\smallskip

We begin the verification of these criteria by applying Theorem \ref{maxprinciple} to $\Phi_m$. Thus for any $t \in \bar{I}_m$ and $v \in \mathcal{B}_m$ we have
\begin{equation}
\begin{aligned}
||\Phi_m(v)(t)|| & \leq  ||\rho(t_{m-1})|| + \int_{t_{m-1}}^t ||f(v+\omega)-f(\omega)|| \dif s + \int_{I_m} ||f(\omega)-f(U)|| \dif s  \\
                    & + \int_{I_m} ||R_T(s)|| \dif s+ \int_{I_m} ||\epsilon_t|| \dif s.
\end{aligned}
\end{equation}
Using \eqref{eq:Lip} with the monotonicity of $\mathcal{L}$ and recalling \eqref{psidef1} we obtain
\begin{equation}
\begin{aligned}
\label{phibound1}
||\Phi_m(v)(t)|| & \leq  \psi_m + \int_{t_{m-1}}^t ||v(s)||\mathcal{L}(s,||v(s)||+||U(s)||+||\epsilon(s)||,||U(s)||+||\epsilon(s)||) \dif s.
\end{aligned}
\end{equation}
In order to bound this further, we require another lemma.
\begin{lemma}
\label{spacecorollary}
For any $t \in \bar{I}_m$, the estimate
\begin{equation}
\begin{aligned}
\notag
||\epsilon(t)|| \leq C_{\infty}\xi_m,
\end{aligned}
\end{equation}
holds where
\begin{equation}
\begin{aligned}
\notag
\xi_m := \max\!\left\{\log(1 \slash \underbar{h}_{m-1})\!\max_{K \in \mathcal{T}_{m-1}} \eta_S^{m-1} \big|_K, \, \log(1 \slash \underbar{h}_m)\! \max_{K \in \mathcal{T}_m} \eta_S^m \big|_K \!\right\}\!.
\end{aligned}
\end{equation}
\end{lemma}
\begin{proof}
Since $\epsilon$ is linear in time, its maximal value occurs at either the left or right end point {\bf --} the result then follows directly from Lemma \ref{spaceest}.
\end{proof}

Applying Lemma \ref{spacecorollary} to \eqref{phibound1}, using the monotonicity of $\mathcal{L}$ and noting that $v \in \mathcal{B}_m$ yields
\begin{equation}
\begin{aligned}
||\Phi_m(v)||_m & \leq  \psi_m\!\left[1+ \delta_m \int_{I_m} L(s, \delta_m) \dif s \right]\!.
\end{aligned}
\end{equation}
where
\begin{equation}
\begin{aligned}
\notag
L(s,\delta) := \mathcal{L}(s,\delta\psi_m + ||U(s)|| + C_{\infty}\xi_m, \delta\psi_m + ||U(s)|| + C_{\infty}\xi_m), \qquad s \in I_m, \,\,\delta \in [1,\infty).
\end{aligned}
\end{equation}
Thus we obtain that property (1) is satisfied if 
\begin{equation}
\tag{$\star$}
1+ \delta_m \int_{I_m} L(s, \delta_m) \dif s \leq \delta_m.
\end{equation}

To show that property (2) holds, let $v_1, v_2 \in \mathcal{B}_m$ then the definition of $\Phi_m$ implies that
\begin{equation}
\begin{aligned}
(\Phi_m(v_1)-\Phi_m(v_2))(t) & =  \int_{t_{m-1}}^t\mathrm{e}^{(t-s) \Delta} [f(v_1+\omega)-f(v_2+\omega)] \dif s, \qquad t \in \bar{I}_m.
\end{aligned}
\end{equation}
Applying Theorem \ref{maxprinciple} and Lemma \ref{spacecorollary} together with the monotonicity property of $\mathcal{L}$ and noting that $v_1, v_2 \in \mathcal{B}_m$ we obtain
\begin{equation}
\begin{aligned}
& ||(\Phi_m(v_1)-\Phi_m(v_2))(t)|| \leq \int_{I_m} ||f(v_1+\omega)-f(v_2+\omega)|| \dif s \\
& \leq \int_{I_m} ||(v_1 - v_2)(s)|| \mathcal{L}(s,||v_1(s)||+||U(s)||+||\epsilon(s)||,||v_2(s)||+||U(s)||+||\epsilon(s)||) \dif s \\
& \leq \int_{I_m} L(s,\delta_m) ||(v_1 - v_2)(s)||\dif s.
\end{aligned}
\end{equation}
Therefore,
\begin{equation}
\begin{aligned}
||\Phi_m(v_1)-\Phi_m(v_2)||_m & \leq ||v_1-v_2||_m \int_{I_m} L(s, \delta_m) \dif s,
\end{aligned}
\end{equation}
which is a contraction if  
\begin{equation}
\tag{$\star\star$}
\int_{I_m} L(s,\delta_m) \dif s < 1.
\end{equation}
Note that $(\star)\!\!\implies\!\!(\star\star$) since coupled with the fact that $\delta_m \in [1,\infty)$, $(\star)$  implies that
\begin{equation}
\int_{I_m} L(s, \delta_m) \dif s \leq 1 - \delta^{-1}_m < 1.
\end{equation}
Thus if ($\star$) is satisfied then by the Banach Fixed Point Theorem, \eqref{erroreqn} has a unique solution $\rho \in \mathcal{B}_m$. As we have \emph{choice} over the value $\delta_m$ can take in $(\star)$ then for practical reasons we choose $\delta_m \in [1,\infty)$ to be, if it exists, the \emph{smallest} root of the function $\phi_m:[1,\infty) \to \mathbb{R}$ defined by
\begin{equation}
\notag
\phi_m(\delta) := 1+ \delta\!\left[\int_{I_m} L(s, \delta) \dif s -1 \right]\!.
\end{equation}

Furthermore, we can in fact obtain a tighter bound on $\rho$ under no additional assumptions. Indeed, we now know that if $\delta_m$ exists that $\Phi_m(\rho) = \rho \in \mathcal{B}_m$ satisfies \eqref{phibound1}; therefore, upon applying Gronwall's inequality and Lemma \ref{spacecorollary} to \eqref{phibound1} we immediately deduce the bound
\begin{equation}
||\rho||_m \leq r_m \psi_m,
\end{equation}
where $r_m \geq 1$ is given by 
\begin{equation}
\notag
r_m = \exp\!\bigg(\int_{I_m} \mathcal{L}(s,\delta_m\psi_m+ ||U(s)|| + C_{\infty}\xi_m,||U(s)|| + C_{\infty}\xi_m) \dif s \bigg).
\end{equation}

\subsection{Computable Error Bound} Now that $\delta_m$ has been defined, we must characterize $\psi_m$ from \eqref{psidef1} in an \emph{a posteriori} fashion in order to obtain a fully computable error bound. To do this, we must first estimate the term $||\rho(t_{m-1})||$ along with the remaining terms containing $\epsilon$ in \eqref{psidef1}.

In the previous subsection, we deduced that $\rho \in \mathcal{B}_m$ if $\delta_m$ exists; however, we assumed \emph{a priori} knowledge of the existence of $\rho(t_{m-1})$ in \eqref{psidef1}. To rectify this, we note that for $m > 1$ we have
\begin{equation}
\label{mrhobound}
||\rho(t_{m-1})|| \leq ||\rho||_{m-1} \leq r_{m-1}\psi_{m-1},
\end{equation}
if $\rho \in \mathcal{B}_{m-1}$. Similarly, $\rho \in \mathcal{B}_{m-1}$ if $\delta_{m-1}$ exists and if we can verify the existence of $\rho(t_{m-2})$. Continuing in this way, we see by recursion that $\rho \in \mathcal{B}_m$ provided that $\delta_1, \ldots, \delta_{m}$ exist; however, this still leaves us with the task of estimating $\rho(0)$ for the first interval. To this end, we rewrite $\rho$ as $\rho = e - \epsilon$ and utilize Lemma \ref{spaceest} to obtain
\begin{equation}
\label{initrhobound}
||\rho(0)|| \leq ||e(0)||+ ||\epsilon(0)|| \leq \eta_I,
\end{equation}
where $\eta_I$ is the \emph{initial condition estimator} given by
\begin{equation*}
\eta_I := ||e(0)|| + C_{\infty}\log(1 \slash \underbar{h}_0)\!\max_{K \in \mathcal{T}_0} \eta_S^0 \big |_K.
\end{equation*}
With regards to the remaining terms in \eqref{psidef1}, the first term containing $\epsilon$ can be estimated directly by using the monotonicity of $\mathcal{L}$ combined with the bound of Lemma \ref{spacecorollary}, viz.,
\begin{equation}
\label{firstepsbound}
\int_{I_m} ||\epsilon(s)||\mathcal{L}(s,||U(s)||,||U(s)||+||\epsilon(s)||) \dif s \leq C_{\infty}\xi_m\!\int_{I_m} \mathcal{L}(s,||U(s)||,||U(s)||+C_{\infty}\xi_m) \dif s.
\end{equation}
To bound the second $\epsilon$ term in \eqref{psidef1}, we require a lemma.
\begin{lemma}\label{secondepsbound}
The following bound holds
\begin{equation*}
\int_{I_m} ||\epsilon_t|| \dif s \leq C_{\infty}{\xi}_m',
\end{equation*}
with
\begin{equation}
\begin{aligned}
\notag
\xi_m' & := \log(1 \slash \,\widehat{\underbar{h}}_m) \, k_m\!\!\max_{\widehat{K} \in \mathcal{T}_{m-1} \wedge \mathcal{T}_m} \max_{\widecheck{K} \subseteq \widehat{K}} {\tightoverset{\bigcdot}{\eta}}_S^m \big|_\widecheck{K},
\end{aligned}
\end{equation}
where $\widehat{h}_m := \min\{\underbar{h}_{m-1}, \underbar{h}_m\}$ and where ${\tightoverset{\bigcdot}{\eta}}_S^m$ is the space derivative estimator given by
\begin{equation}
\begin{aligned}
\notag
\hspace{-2.5mm}{\tightoverset{\bigcdot}{\eta}}_S^m \big|_\widecheck{K} := h^2_{\widehat{K}} k_m^{-1}a^{-1}||A^m-A^{m-1}+a\Delta (U^m-U^{m-1})||_{L^{\infty}(\widecheck{K})} + h_{\widehat{K}}k^{-1}_m ||\jump{\nabla(U^m-U^{m-1})}||_{L^{\infty}(\partial \widecheck{K} \backslash \partial \Omega)},
\end{aligned}
\end{equation}
for $\mathcal{T}_{m-1} \vee \mathcal{T}_m \ni \widecheck{K} \subseteq \widehat{K} \in \mathcal{T}_{m-1} \wedge \mathcal{T}_m$. As before, ${C}_\infty$ is a constant that is independent of the maximum mesh-size, $a$, $u$ and $U$ but may be dependent upon the size of the domain $\Omega$ as well as the number of refinement levels between $\mathcal{T}_{m-1}$ and $\mathcal{T}_m$.
\end{lemma}
\begin{proof}
We note the observation made in Corollary 2.9 of \cite{DLM09} that $\epsilon_t$ is Galerkin orthogonal to the space $\mathcal{V}_h^{m-1} \cap \mathcal{V}_h^{m}$. The stated bound then follows by conducting the error analysis as in \cite{DK14} but with a quasi-interpolant based on the mesh $\mathcal{T}_{m-1}\wedge\mathcal{T}_m$.
\end{proof}

Finally, applying the bounds of \eqref{mrhobound}, \eqref{initrhobound}, \eqref{firstepsbound} and Lemma \ref{secondepsbound} to the left-hand side of \eqref{psidef1}, we see that we can define $\psi_m$ to be the computable \emph{a posteriori} quantity given by
\begin{equation*}
\psi_m  := \begin{cases}\displaystyle \eta_I +  C_{\infty}\xi_m\!\int_{I_m} \mathcal{L}(s,||U(s)||,||U(s)||+C_{\infty}\xi_m) \dif s +\eta_T^m+C_\infty\xi_m'\quad & \text{if } m = 1 \\
\displaystyle r_{m-1}\psi_{m-1}+ C_{\infty}\xi_m\!\int_{I_m} \mathcal{L}(s,||U(s)||,||U(s)||+C_{\infty}\xi_m) \dif s+\eta_T^m+C_\infty\xi_m' \quad & \text{if } m \neq 1 \end{cases}.
\end{equation*}

With $\psi_m$ defined, all components of the error bound are now in place as well as fully computable and we are ready to state the main result.
\begin{theorem}
\label{maintheorem}
Suppose that $\delta_1, \ldots, \delta_M$ exist then the $L^{\infty}L^{\infty}$ error of the IMEX method \eqref{IMEX} satisfies the a posteriori bound
\begin{equation}
\notag
\max_{1 \leq m \leq M} ||e||_m \leq r_M \psi_M + C_{\infty}\max_{0 \leq m \leq M} \log(1 \slash \underbar{h}_m)\! \max_{K \in \mathcal{T}_m} \eta_S^m \big|_K.
\end{equation}
\end{theorem}
\begin{proof}
Since $\delta_1,\ldots,\delta_M$ exist then by the exposition of the previous subsection we have
\begin{equation}
\notag
||\rho||_m \leq r_m\psi_m,
\end{equation}
for any $1 \leq m \leq M$. Noting the decomposition $e = \rho + \epsilon$, we obtain that
\begin{equation}
\notag
\max_{1 \leq m \leq M} ||e||_m \leq \max_{1 \leq m \leq M} ||\rho||_m + \max_{1 \leq m \leq M} ||\epsilon||_m \leq r_M \psi_M + \max_{1 \leq m \leq M} ||\epsilon||_m.
\end{equation}
The stated result then follows from Lemma \ref{spacecorollary}.
\end{proof}

Given these results, a natural question to ask is whether $\delta_m \in [1,\infty)$, the smallest root of $\phi_m$, can actually exist at all. This is the focus of our next lemma.
\begin{lemma}
If the time step length $k_m$ is small enough then $\phi_m$ has a root in $[1,\infty)$.
\end{lemma}
\begin{proof}
We omit full details of the proof for brevity but remark that the proof is simple and essentially the same as that of Lemma 3.28 in \cite{KMW16}.
\end{proof}

We conclude this section by showing that the \emph{a posteriori} error bound of Theorem \ref{maintheorem} can be vastly simplified when $f$ is independent of $u$.
\begin{corollary}\label{findeptheorem}
Suppose that $f$ is independent of $u$ then the error of the IMEX method \eqref{IMEX} unconditionally satisfies the a posteriori bound
\begin{equation}
\notag
\max_{1 \leq m \leq M} ||e||_m \leq ||e(0)|| + \sum_{m=1}^M \eta_T^m + C_\infty \sum_{m=1}^M \xi_m'+ C_{\infty}\max_{0 \leq m \leq M} \log(1 \slash \underbar{h}_m)\! \max_{K \in \mathcal{T}_m} \eta_S^m \big|_K.
\end{equation}
\end{corollary}
\begin{proof}
Since $f$ is independent of $u$, it follows from \eqref{eq:Lip} that $\mathcal{L} = 0$. We recall that $\delta_m$ is the smallest root of the function $\phi_m : [1,\infty) \to \mathbb{R}$ given by 
\begin{equation}
\notag
\phi_m(\delta) = 1+ \delta\!\left[\int_{I_m} L(s, \delta) \dif s -1 \right]\! = 1 - \delta.
\end{equation}
Therefore, $\delta_m = 1$ regardless of the size of the time step length $k_m$ and so the \emph{a posteriori} error bound of Theorem \ref{maintheorem} holds unconditionally. The stated result then follows from Theorem \ref{maintheorem} by applying $\psi_m$ recursively and upon noting that $r_m = 1$.
\end{proof}

\begin{remark}
The \emph{a posteriori} error bound of Corollary \ref{findeptheorem} is essentially that of Theorem 4.2 in \cite{DLM09} but with a sharper time estimator.
\end{remark}

\section{Adaptive Algorithms}

In this section, we propose an adaptive algorithm that is based on the idea of minimization of the \emph{a posteriori} error bound of Theorem \ref{maintheorem}. Ultimately, our goal is an adaptive algorithm that is applicable to both blow-up and fixed-time problems. Here, we consider an adaptive strategy that is based on using the residuals to control the time step lengths and mesh sizes but we emphasize that other adaptive strategies are possible for blow-up problems. For example, in \cite{HW16}, an existence analysis is carried out for implicit approximations to \eqref{model_strong} via fixed point arguments and the results are used to select the length of the time steps in the scheme. A similar adaptive strategy could be used here in the sense that we could continue to reduce the size of the time step on $I_m$ until $\delta_m$ exists and then fix this time step length before moving on to the next interval. Indeed, it was shown in \cite{HW16} that choosing the size of the time steps in a way analagous to this can lead to superconvergence to the blow-up time; however, such a strategy does also come with certain disadvantages.  Firstly, it is unclear how to generalize this approach to the case where $f$ is independent of $u$ since then $\delta_m$ \emph{always} exists. Moreover, it is unclear what a `natural' stop critereon for such an adaptive strategy should be whereas choosing the time step lengths according to the size of the residuals allows the non-existence of $\delta_m$ to be the stop critereon. 

We contend that a general adaptive algorithm based on the \emph{a posteriori} error bound of Theorem \ref{maintheorem} should  revert to a reasonable, well-known adaptive strategy when applied to a simple case such as that of Corollary \ref{findeptheorem}. For this reason, we first outline an adaptive algorithm based on Corollary \ref{findeptheorem} and then extend this algorithm in a logical fashion to incorporate the full bound of Theorem \ref{maintheorem}. In what follows, we give a general outline of the idea behind the adaptive algorithms but for brevity we also include the pseudocode of the full algorithm in Algorithm 1.

\subsection{Adaptive Algorithm for Corollary \ref{findeptheorem}}

The nature of the data in Corollary \ref{findeptheorem} means that the algorithm in this subsection will be for a \emph{fixed-time problem}. As such, the inputs to the algorithm include the data, the domain $\Omega$, the final time $T$, a coarse initial mesh $\mathcal{T}_0$ of $\Omega$ and an unrefined initial time step length $k_1 \leq T$.

Suppose that we are on the generic time step $m > 1$ then the backward solution $U^{m-1} \in \mathcal{V}_h^{m-1}$ and the mesh $\mathcal{T}_{m-1}$ are already fixed. To proceed on the current interval, we first set the mesh $\mathcal{T}_m = \mathcal{T}_{m-1}$ and then calculate the forward solution $U^m \in \mathcal{V}_h^m$ given by \eqref{IMEX}. The simple structure of the error bound immediately suggests defining the refinement indicators
\begin{equation}
\begin{aligned}
\notag
&{\tt ref}_T^m := \eta_T^m,  \qquad \qquad\qquad & {\tt ref}_S^m \big|_K := \max\!\Big\{{\tightoverset{}{\eta}}_S^m \big|_K, \,{\tightoverset{\bigcdot}{\eta}}_S^m \big |_K\Big\}, \,\,K \in \mathcal{T}_m.
\end{aligned}
\end{equation}
Note that $\displaystyle {\tt ref}_T^m$ is local to each time step while $\displaystyle {\tt ref}_S^m$ is local to each mesh element. As is standard for spatial adaptivity done via $L^\infty$ norm error estimates, we ignore the global logarithmic terms \cite{NSSV06}. To control the size of the time steps and the mesh elements, we introduce four tolerances: a \emph{spatial refinement tolerance} $\tt stol^+$, a \emph{spatial coarsening tolerance} $\tt stol^-$, a \emph{temporal refinement tolerance} $\tt ttol^+$ and a \emph{temporal coarsening tolerance} $\tt ttol^-$. If necessary, we begin by either refining or coarsening the time step length $k_m$ and recalculating the forward solution $U^m \in \mathcal{V}_h^m$ until 
\begin{equation}
\notag
{\tt ttol^-} \leq {\tt ref}_T^m \leq {\tt ttol^+},
\end{equation}
is satisfied. We then fix this time step length.  Next, we proceed spatially by refining all elements $K \in \mathcal{T}_m$ such that ${\tt ref}_S^m \big|_K > {\tt stol^+}$ and coarsening all elements such that ${\tt ref}_S^m \big|_K < {\tt stol^-}$. We then recalculate (if necessary) and fix the forward solution $U^m \in \mathcal{V}_h^m$. After this is done, we set $k_{m+1} = k_{m}$ and proceed to the next interval unless the total time $t_{m+1} = t_m + k_{m+1}$ would surpass the final time $T$  in which case we set $k_{m+1} = T - t_{m}$. When the final time is reached, we halt our computations.

All that remains is to deal with the coarse grid and time step length on the first interval. To do this, we first modify the space refinement indicator to account for the term $\eta_I$, viz.,
\begin{equation}
\notag
{\tt ref}_S^1 \big|_K := \max\!\Big\{||e(0)||_{L^{\infty}(K)},\,{\tightoverset{}{\eta}}_S^0 \big|_K,\,{\tightoverset{}{\eta}}_S^1 \big|_K,\, {\tightoverset{\bigcdot}{\eta}}_S^1 \big |_K\Big\}, \,\,K \in \mathcal{T}_1.
\end{equation}
To begin, we set the mesh $\mathcal{T}_1 = \mathcal{T}_0$ and calculate the backward solution $U^0 \in \mathcal{V}_h^0$ given by \eqref{IC} and the forward solution $U^1 \in \mathcal{V}_h^1$ given by \eqref{IMEX}. We then proceed, via the tolerance strategy outlined above, by refining (or coarsening) the time step length $k_1$ and the mesh $\mathcal{T}_1 = \mathcal{T}_0$ concurrently then recalculating the backward and forward solutions $U^0 \in \mathcal{V}_h^0$ and $U^1 \in \mathcal{V}_h^1$ until both ${\tt ref}_T^1 \leq {\tt ttol^+}$  and $\displaystyle \max_{K \in \mathcal{T}_1} {\tt ref}_S^1 \big|_K \leq {\tt stol^+}$ are satisfied. After this is done, we set $k_2 = k_1$ and proceed to the next interval.

\begin{remark}
The refinement and coarsening tolerances need to be chosen sufficiently far apart so that the finite element solution does not get caught in an infinite refine and coarsen loop. 
\end{remark}
\subsection{Adaptive Algorithm for Theorem \ref{maintheorem}}

In this subsection, we modify the adaptive algorithm of the previous subsection for adaptivity under the full estimator of Theorem \ref{maintheorem}; we do this in such a way that the algorithm of the previous subsection is recovered when $f$ is independent of $u$. In fact, the only changes that need to be made to the algorithm of the previous subsection are:
\smallskip

(1) The stop critereon must be altered to account for the fact that the error bound may now not necessarily hold.
\smallskip

(2) The refinement indicators must be modified for adaptivity under the full error estimator.
\smallskip

To address point (1), we recall from Theorem \ref{maintheorem} that the error bound holds on $I_m$ provided that $\delta_1, \ldots, \delta_m$ exist; therefore, after all the adaptive procedures on the current interval are complete, we attempt to calculate $\delta_m$ via a root-finding algorithm. If we find a root then we continue to the next interval (unless the final time is reached); if not, we halt our computations. 

A na\"{\i}ve approach towards addressing point (2) would be to simply use the refinement indicators of the previous subsection for adaptivity, however, doing so not only completely ignores the structure of the error estimator but also the length scales inherent to a blow-up problem. Consequently, such a refinement indicator causes excessive over-refinement close to the blow-up time which, in turn, results in suboptimal convergence \cite{CGKM16,M15,KMW16}. In order to characterize the new refinement indicators, we define
\begin{equation}
\begin{aligned}
\notag
& \widetilde{r}_0 := 1,  \qquad \qquad\qquad & \displaystyle\widetilde{r}_m := \prod_{i = 1}^m r_i,\,\, m > 1,
\end{aligned}
\end{equation}
to be the accumulation of the values $r_i$, $1 \leq i \leq m$. We remark that $\widetilde{r}_m$ is a value intimately connected with the rate of blow-up of the exact solution on the interval $(0,t_m)$ as shown in the numerical experiments of \cite{KMW16} and as we shall show here in the sequel; therefore, it is natural that this quantity should appear in our refinement indicators. With this notation at hand, the temporal refinement indicator is given by
\begin{equation}
\begin{aligned}
\notag
&{\tt ref}_T^m := \widetilde{r}^{\,-1}_{m-1}\eta_T^m.
\end{aligned}
\end{equation}
For a rigorous justification of this choice, we refer the reader to \cite{KMW16}. Defining the space refinement indicator is more involved as reviewing the error bound of Theorem \ref{maintheorem}, it is clear that some terms are affected by $r_{m}$ while others are not. The lone spatial term outside the recursive portion of the error bound is independent of ${r}_m$ which suggests that we demand that 
\begin{equation}
\begin{aligned}
\label{spaceadap1}
{\tightoverset{}{\eta}}_S^m \big|_K \leq {\tt stol^+}, \,\, K \in \mathcal{T}_m,
\end{aligned}
\end{equation}
is satisfied. On the other hand, the space derivative estimator ${\tightoverset{\bigcdot}{\eta}}_S^m$ is part of the term $\psi_m$ which is affected by ${r}_m$ and so we ask that
\begin{equation}
\begin{aligned}
\label{spaceadap2}
\widetilde{r}_{m-1}^{\,-1}{\tightoverset{\bigcdot}{\eta}}_S^m \big |_K \leq {\tt stol^+}, \,\, K \in \mathcal{T}_m,
\end{aligned}
\end{equation}
is satisfied. The remaining spatial term in $\psi_m$ is affected by ${r}_m$ as well but we must also divide it by the time step length $k_m$ in order to incorporate it into the space refinement indicator as the space refinement indicator shouldn't be (strongly) dependent upon the time step length. Therefore, we require that
\begin{equation}
\begin{aligned}
\label{spaceadap3}
\left[\,\widetilde{r}^{\,-1}_{m-1}k^{-1}_m \int_{I_m} \mathcal{L}(s,||U(s)||,||U(s)||+\xi_m) \dif s \right]\!{\tightoverset{}{\eta}}_S^m \big|_K \leq {\tt stol^+}, \,\, K \in \mathcal{T}_m,
\end{aligned}
\end{equation}
is satisfied. So upon defining
\begin{equation}
\begin{aligned}
\notag
\alpha_m := \max\!\left\{1, \, \widetilde{r}^{\,-1}_{m-1}k^{-1}_m \int_{I_m} \mathcal{L}(s,||U(s)||,||U(s)||+\xi_m) \dif s\right\}\!,
\end{aligned}
\end{equation}
we combine \eqref{spaceadap1}, \eqref{spaceadap2} and \eqref{spaceadap3} to define the space refinement indicator (for $m > 1$), viz.,
\begin{equation}
\begin{aligned}
\notag
{\tt ref}_S^m \big|_K := \max\!\Big\{\alpha_m{\tightoverset{}{\eta}}_S^m \big|_K, \,\widetilde{r}^{\,-1}_{m-1}{\tightoverset{\bigcdot}{\eta}}_S^m \big |_K\Big\}, \,\,K \in \mathcal{T}_m.
\end{aligned}
\end{equation}
As in the previous subsection, we must also modify the space refinement indicator on the first interval in order to take $\eta_I$ into account so we set
\begin{equation}
\begin{aligned}
\notag
{\tt ref}_S^1 \big|_K := \max\!\Big\{||e(0)||_{L^{\infty}(K)},\, \alpha_1{\tightoverset{}{\eta}}_S^0 \big|_K,\, \alpha_1{\tightoverset{}{\eta}}_S^1 \big|_K, \,{\tightoverset{\bigcdot}{\eta}}_S^1 \big |_K\Big\}, \,\,K \in \mathcal{T}_1.
\end{aligned}
\end{equation}

Finally, we recall from Corollary \ref{findeptheorem} that for $f$ independent of $u$ we have $\mathcal{L} = 0$ and $\delta_m = 1$. As $\delta_m$ always exists for this type of data, the additional stop critereon introduced in this subsection never comes into play. Moreover, $\mathcal{L} = 0 \!\implies\! \widetilde{r}_m = 1$ and so the reference indicators introduced here for the general case devolve into those of the previous subsection for $f$ independent of $u$. Therefore, we conclude that the two adaptive algorithms are the same for this type of data.

\begin{algorithm}
  \begin{algorithmic}[1]
     \State {\bf Input:} $a$, $f$, $u_0$, $\Omega$, $T$, $\mathcal{T}_0 \,(= \!\mathcal{T}_1)$, $k_1$, ${\tt stol^+}$, ${\tt stol^-}$, ${\tt ttol^+}$, ${\tt ttol^-}$.
     \State Compute $U^0$.
 \State Compute $U^1$ from $U^0$.
  
    \While {$\displaystyle{\tt ref}_T^1 > {\tt ttol^+} \,\text{\bf or}\, \max_{K \in \mathcal{T}_0} {\tt ref}_S^1 \big |_K > {\tt stol^+}$}
        \State {Modify $\mathcal{T}_0 \,(= \!\mathcal{T}_1)$ by refining all elements such that ${\tt ref}_S^1 \big |_K > {\tt stol^+}$ and coarsening all elements such that ${\tt ref}_S^1 \big |_K < {\tt stol^-}$.}
   \If   {$\displaystyle{\tt ref}_T^1 > {\tt ttol^+}$}
 \State $k_{1} \leftarrow k_{1}/2$.
    \EndIf
     \State Compute $U^0$.
 \State Compute $U^1$ from $U^0$.
    \EndWhile
    
   \State Attempt to compute $\delta_1$.
 \State  Set $m =  0$.

\While {$\delta_{m+1}$ exists \,{\bf and}\, $t_{m+1} < T$}
\State $m \leftarrow m+1$. 
 \State Multiply ${\tt stol^+}$, ${\tt stol^-}$, ${\tt ttol^+}$ and ${\tt ttol^-}$ by the factor $\delta_{m}$.
  \State Set $\mathcal{T}_{m+1} = \mathcal{T}_m$ and $k_{m+1} = \min \{k_m, \,T - t_m \}$.
\State Compute $U^{m+1}$ from $U^m$.
        \If {${\tt ref}_T^{m+1}> {\tt ttol^+}$}  
 \State $k_{m+1} \leftarrow k_{m+1}/2$.
\State Compute $U^{m+1}$ from $U^m$.
    \EndIf
        \If {${\tt ref}_T^{m+1} < {\tt ttol^-}$}
 \State $k_{m+1} \leftarrow \min\{2k_{m+1},\,T - t_m\}$.
\State Compute $U^{m+1}$ from $U^m$.
    \EndIf    
    \State Modify $\mathcal{T}_{m+1}$ by refining all elements such that ${\tt ref}_S^{m+1} \big |_K > {\tt stol^+}$ and coarsening all elements such that ${\tt ref}_S^{m+1} \big |_K < {\tt stol^-}$.
    \State Compute $U^{m+1}$ from $U^m$.
        \State Attempt to compute $\delta_{m+1}$.
\EndWhile
\State {\bf Output:} $m$, $t^m$, $U$.
  \end{algorithmic}
  \caption{Space-time adaptivity for the semilinear heat equation}
\end{algorithm}

\section{Numerical Experiments}

We consider an implementation of the adaptive algorithm of the previous section through an application of the {\tt deal.II} finite element library \cite{BHK07, BDHHKKMTW15}. In order to facilitate a comparison between the $L^\infty L^\infty$ estimator of Theorem \ref{maintheorem} and the $L^2H^1$ estimator of \cite{CGKM16}, we consider Example 1 and Example 3 of \cite{CGKM16} but under the  adaptive algorithm of the previous section and driven by the $L^\infty L^\infty$ \emph{a posteriori} error bound derived in this paper. If the \emph{a posteriori} error bound of Theorem \ref{maintheorem} is robust with respect to the distance from the blow-up time then for sufficiently small ${\tt stol^+}$ and ${\tt stol^-}$, we would expect to observe that
\begin{equation}
\begin{aligned}
\notag
|T_{\infty} - T({\tt ttol^+}, N)| \propto N^{-1},
\end{aligned}
\end{equation}
where $T_{\infty}$ is the blow-up time of \eqref{model_strong} and $T$ is the final time produced by the adaptive algorithm in $N$ total time steps under a given temporal refinement tolerance ${\tt ttol^+}$ as this is what was observed in the ODE experiments of \cite{CGKM16, KMW16, M15}. Additionally, we also apply the adaptive algorithm to a nonlinear fixed-time problem in order to demonstrate its generality and to show that the estimator of Theorem \ref{maintheorem} is of optimal order in space and time.

\subsection{Example 1}

Let $\Omega = (-8,8)^2$, $a = 1$, $f(u) = u^2$ and choose the initial condition to be the Gaussian blob given by $u_0(x,y) = 10\exp(-2x^2-2y^2)$. The blow-up set for this example consists of only a single point (the origin) making it spatially uncomplicated which allows us to focus solely on the temporal asymptotics. Now, since $f(u) = u^2$ then for any $v_1, v_2 \in \mathbb{R}$ we have 
\begin{equation}
\begin{split}
|f(v_1)-f(v_2)| = |v_1^2 - v_2^2 | \leq |v_1 - v_2|(|v_1|+|v_2|).
\end{split}
\end{equation}
Therefore, we have $\mathcal{L}(|v_1|,|v_2|) = |v_1| + |v_2|$ in \eqref{eq:Lip} and so $\delta_m$ (if it exists) is the smallest root of the function $\phi_m:[1,\infty) \to \mathbb{R}$ given by
\begin{equation}
\begin{split}
\label{deltaquad}
\phi_m(\delta) = 1+ \delta\!\left[\int_{I_m} L(s, \delta) \dif s -1 \right]\! = 1 + \delta\!\left[2C_{\infty}k_m\xi_m + 2\int_{I_m} ||U(s)|| \dif s -1  \right]+2k_m\psi_m\delta^2.
\end{split}
\end{equation}
In this case, we can calculate $\delta_m$ explicitly via the quadratic formula and so there is no need to use a root finding algorithm here.

Given that we wish to observe the temporal asymptotics and since for this example not much spatial resolution is required, we opt to use polynomials of degree nine. We begin by first setting a small spatial refinement tolerance $\tt stol^+$ so that the spatial error is negligible; we then gradually reduce the temporal refinement tolerance $\tt ttol^+$ in order to observe the rate of convergence to the blow-up time. We include the results in the left-hand side of Table \ref{blowupdata1} alongside the results that utilize the $L^2H^1$ estimator of \cite{CGKM16} on the right-hand side.

\begin{table}[ht]

\caption{\text{Example 1: $L^\infty L^\infty$ estimator of Theorem \ref{maintheorem} (left) and $L^2H^1$ estimator of \cite{CGKM16} (right).}} 
\centering 
\begin{tabular}{c c c c} 
\hline\hline 
${\tt ttol}^+$ & Time Steps & Final Time & $||U(T)||$ \\ 
\hline 
0.25 & 2 & 0.05375 & 11.042 \\ 
$0.25^2$ & 8 & 0.10750 & 13.644 \\
$0.25^3$ & 23 & 0.15453 & 19.936 \\
$0.25^4$ & 52 & 0.17469 & 27.721 \\
$0.25^5$ & 114 & 0.19148 & 42.960 \\
$0.25^6$ & 493 & 0.20702 & 103.901 \\
$0.25^7$ & 1004 & 0.21080 & 164.944 \\
$0.25^8$ & 2031 & 0.21332 & 273.236 \\
$0.25^9$ & 4093 & 0.21490 & 458.924 \\
$0.25^{10}$ & 8218 & 0.21571 & 745.826 \\
$0.25^{11}$ & 16479 & 0.21625 & 1276.960 \\

\hline 
\end{tabular}\qquad\qquad
\begin{tabular}{c c c} 
\hline\hline 
Time Steps & Final Time & $||U(T)||$ \\ 
\hline 
3 & 0.09375 & 12.244 \\ 
8 & 0.12500 & 14.742 \\
19 & 0.14844 & 18.556 \\
42 & 0.16406 & 23.468 \\
92 & 0.17969 & 32.108 \\
195 & 0.19043 & 44.217 \\
405 & 0.19775 & 60.493 \\
832 & 0.20313 & 83.315 \\
1698 & 0.20728 & 117.780 \\
3443 & 0.21014 & 165.833 \\
6956 & 0.21228 & 238.705 \\
14008 & 0.21375 & 343.078 \\
28151 & 0.21478 & 496.885 \\
56489 & 0.21549 & 722.884 \\
\hline 
\end{tabular}

\label{blowupdata1}
\end{table}

The results show that for this example the $L^\infty L^\infty$ estimator of Theorem \ref{maintheorem} outperforms the $L^2H^1$ estimator of \cite{CGKM16} in terms of rate of convergence to the blow-up time. We recall from \cite{CGKM16} that given two consecutive data points we can approximate the exact blow-up time $T_{\infty}$ as follows
\begin{equation}
\begin{split}
\label{blowuptimeapprox}
T_{\infty} \approx \frac{t_m||U^m||-t_{m-1}||U^{m-1}||}{||U^m||-||U^{m-1}||}.
\end{split}
\end{equation}
Applying this here, we obtain the approximation $T_{\infty} \approx 0.217015$. Using this approximation to $T_{\infty}$, we take the data from Table \ref{blowupdata1} and plot the distance from the blow-up time $|T-T_{\infty}|$ versus the total number of time steps $N$ in Figure \ref{example1results}. The plot shows that for this example we have
\begin{equation}
\begin{aligned}
\notag
|T_{\infty} - T({\tt ttol^+}, N)| \propto N^{-3 \slash 4}.
\end{aligned}
\end{equation}
This is slightly slower than expected given the ODE results of \cite{CGKM16, KMW16, M15}. Finally, for the final computational run we plot the magnitude of the numerical solution $||U(t)||$, the parabolic estimator $r_m \psi_m$ and the value $m\widetilde{r}_m{\tt ttol}^+$ versus the inverse of the distance to the blow-up time. From the results, given in Figure \ref{example1results}, we deduce the asymptotic estimate
\begin{equation}
\begin{aligned}
\notag
||U(t)|| \propto |t - T_{\infty}|^{-1},
\end{aligned}
\end{equation}
which is consistent with the asymptotics of the exact solution \cite{MZ98,MZ00} suggesting that our numerical solution is reasonable. Moreover, we expect (cf. Corollary 4.3 of \cite{KMW16}) that in the spatially asymptotic regime and under the adaptive strategy induced by the adaptive algorithm that the parabolic error $\rho$ satisfies
\begin{equation}
\begin{aligned}
\notag
 \max_{1 \leq k \leq m} ||\rho||_k \leq r_m \psi_m \leq m\widetilde{r}_m {\tt ttol}^+,
\end{aligned}
\end{equation}
which we confirm in Figure \ref{example1results}. Therefore, upon observing that the gradient of the estimator curve in Figure \ref{example1results} is two, we deduce that there exists a constant $C > 0$ that is independent of the distance to the blow-up time and the maximum time step length such that 
\begin{equation}
\begin{aligned}
\notag
 \max_{1 \leq k \leq m} ||\rho||_k \leq Cm|t_m - T_{\infty}|^{-2} \,{\tt ttol}^+.
\end{aligned}
\end{equation}

\begin{figure}
\centering
\includegraphics[scale=0.265]{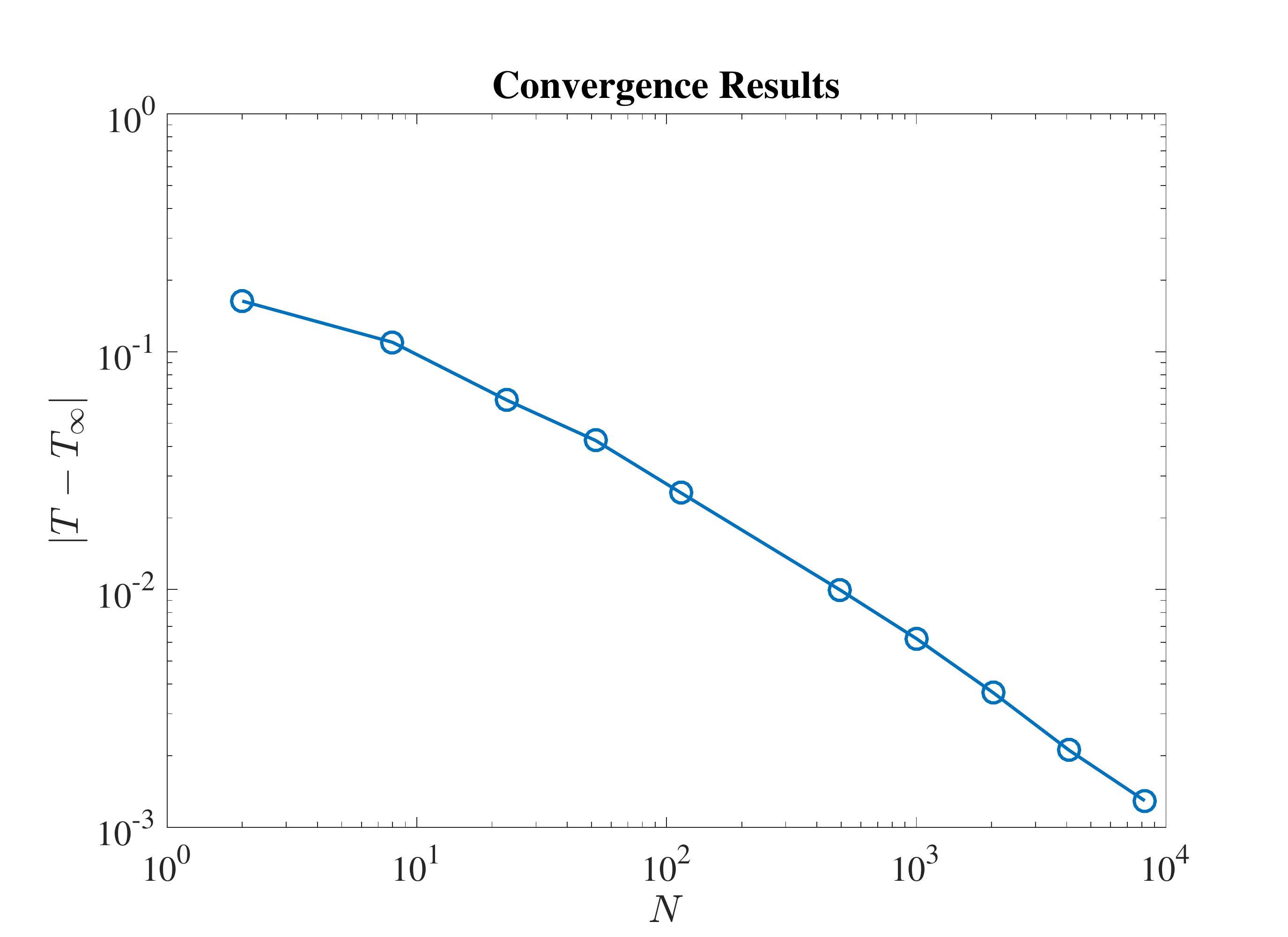} \includegraphics[scale=0.265]{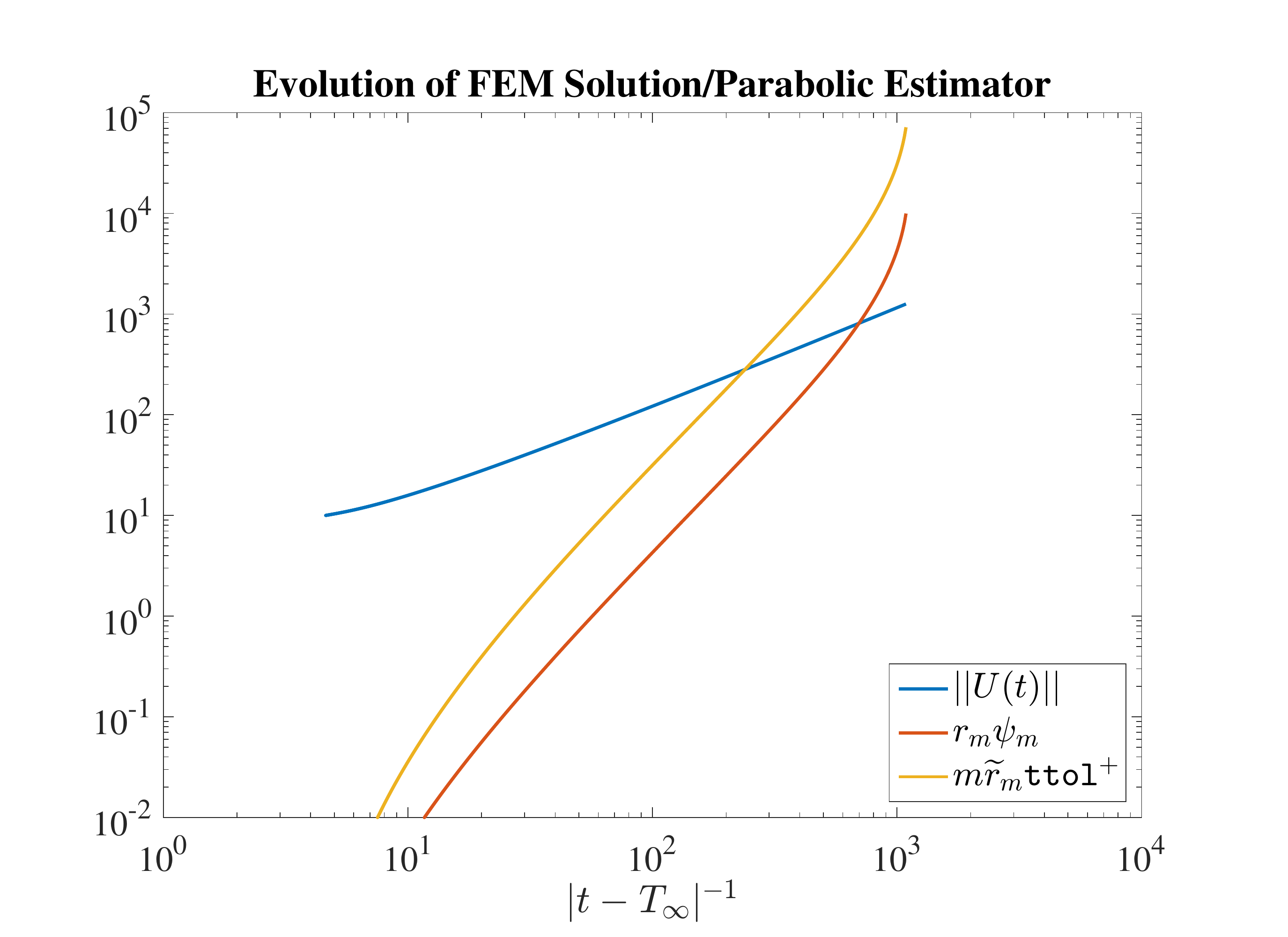}
\caption{\mbox{Example 1: convergence results (left) and evolution of the numerical solution (right).}}
\label{example1results}
\end{figure}

\subsection{Example 2}

Let $\Omega = (-8,8)^2$, $a = 1$, $f(u) = u^2$ and the ``volcano'' type initial condition be given by $u_0(x,y) = 10(x^2+y^2)\exp(-0.5x^2-0.5y^2)$. The blow-up set for this example is a circle centered on the origin making this example a good test of the spatial capabilities of the adaptive algorithm as many degrees of freedom are required in order to resolve the one-dimensional singularity close to the blow-up time. We remark that as the nonlinearity here is the same as in Example 1, $\delta_m$ is again the smallest root of \eqref{deltaquad}.

For this example, we opt to use polynomials of degree six as a compromise {\bf --} this is because we desire a large polynomial degree early on in order to take advantage of when the solution is smooth but close to the blow-up time we would like polynomials of low degree so that we don't get overwhelmed by degrees of freedom. We proceed as in the previous example by choosing a small spatial refinement tolerance $\tt stol^+$ so that the spatial error is negligible; we then gradually reduce the temporal refinement tolerance $\tt ttol^+$ in order to observe the rate of convergence to the blow-up time. The results, displayed on the left-hand side of Table \ref{blowupdata2} alongside the results that utilize the $L^2H^1$ estimator of \cite{CGKM16} on the right-hand side, show that the $L^\infty L^\infty$ estimator of Theorem \ref{maintheorem} again outperforms the $L^2H^1$ estimator of \cite{CGKM16} by an order of magnitude.

\begin{table}[ht]
\caption{\text{Example 2: $L^\infty L^\infty$ estimator of Theorem \ref{maintheorem} (left) and $L^2H^1$ estimator of \cite{CGKM16} (right).}} 
\centering 
\begin{tabular}{c c c c} 
\hline\hline 
${\tt ttol}^+$ & Time Steps & Final Time & $||U(T)||$ \\ 
\hline 
0.25 & 3 & 0.06210 & 10.347 \\ 
$0.25^2$ & 11 & 0.11385 & 17.687 \\
$0.25^3$ & 26 & 0.13455 & 27.548 \\
$0.25^4$ & 71 & 0.15525 & 65.557 \\
$0.25^5$ & 159 & 0.16043 & 119.261 \\
$0.25^6$ & 332 & 0.16301 & 208.434 \\
$0.25^7$ & 710 & 0.16495 & 445.018 \\
$0.25^8$ & 1463 & 0.16556 & 778.815 \\
$0.25^9$ & 2973 & 0.16598 & 1467.920 \\
$0.25^{10}$ & 6115 & 0.16627 & 3340.330 \\
$0.25^{11}$ & 12329 & 0.16635 & 6171.900 \\
$0.25^{12}$ & 24880 & 0.16640 & 11022.400 \\
\hline 
\end{tabular}\qquad\qquad
\begin{tabular}{c c c} 
\hline\hline 
Time Steps & Final Time & $||U(T)||$ \\ 
\hline 
3 & 0.06250 & 10.371 \\ 
10 & 0.09375 & 14.194 \\
36 & 0.11979 & 21.842 \\
86 & 0.13412 & 31.446 \\
190 & 0.14388 & 45.122 \\
404 & 0.15072 & 64.907 \\
880 & 0.15601 & 98.048 \\
1853 & 0.15942 & 146.162 \\
3831 & 0.16176 & 219.423 \\
7851 & 0.16336 & 332.849 \\
16137 & 0.16442 & 505.236 \\
32846 & 0.16512 & 769.652 \\
66442 & 0.16558 & 1175.210 \\
\hline 
\end{tabular}
\label{blowupdata2}
\end{table}

\begin{figure}[h]
\centering
\includegraphics[scale=0.265]{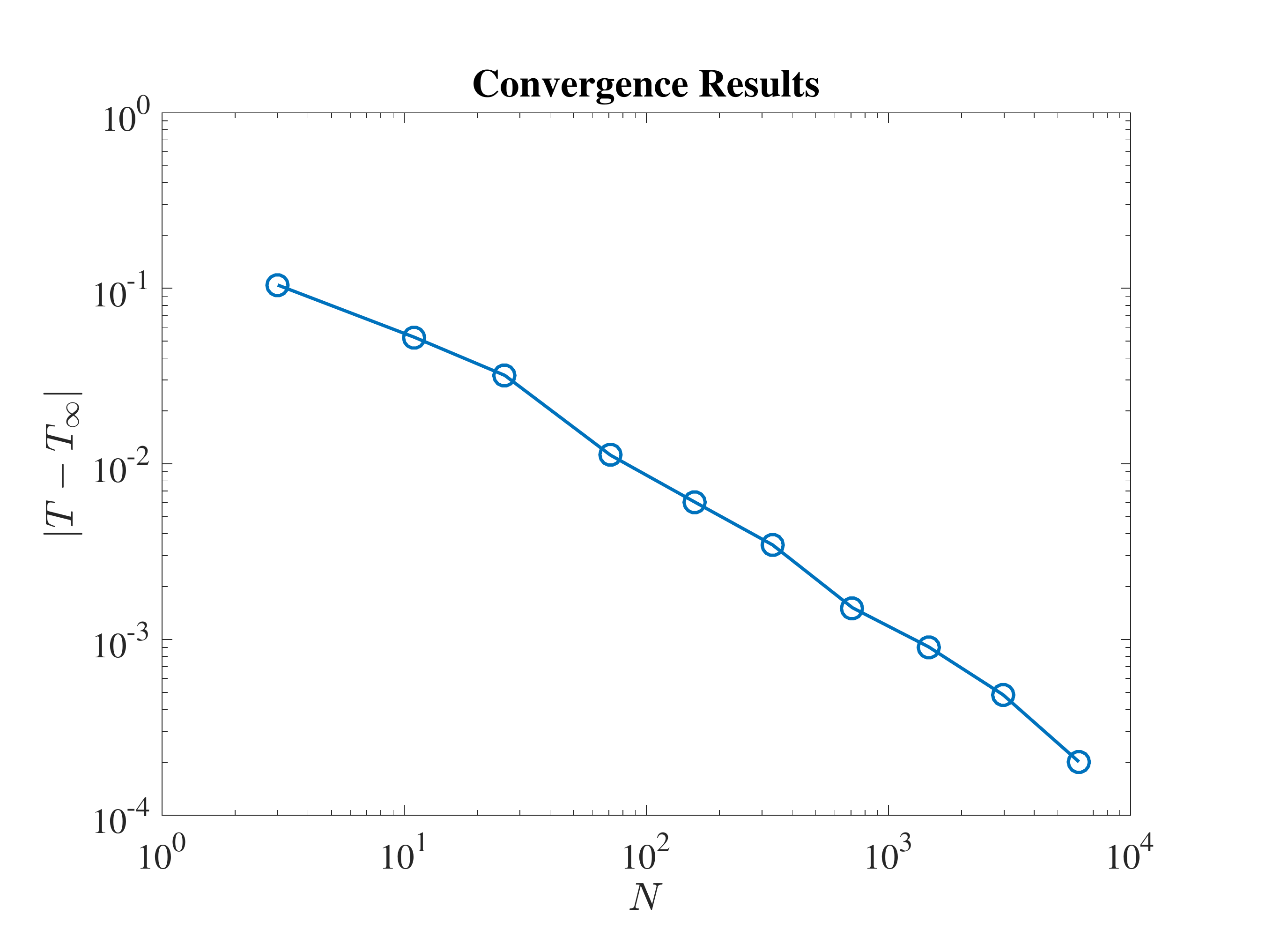} \includegraphics[scale=0.265]{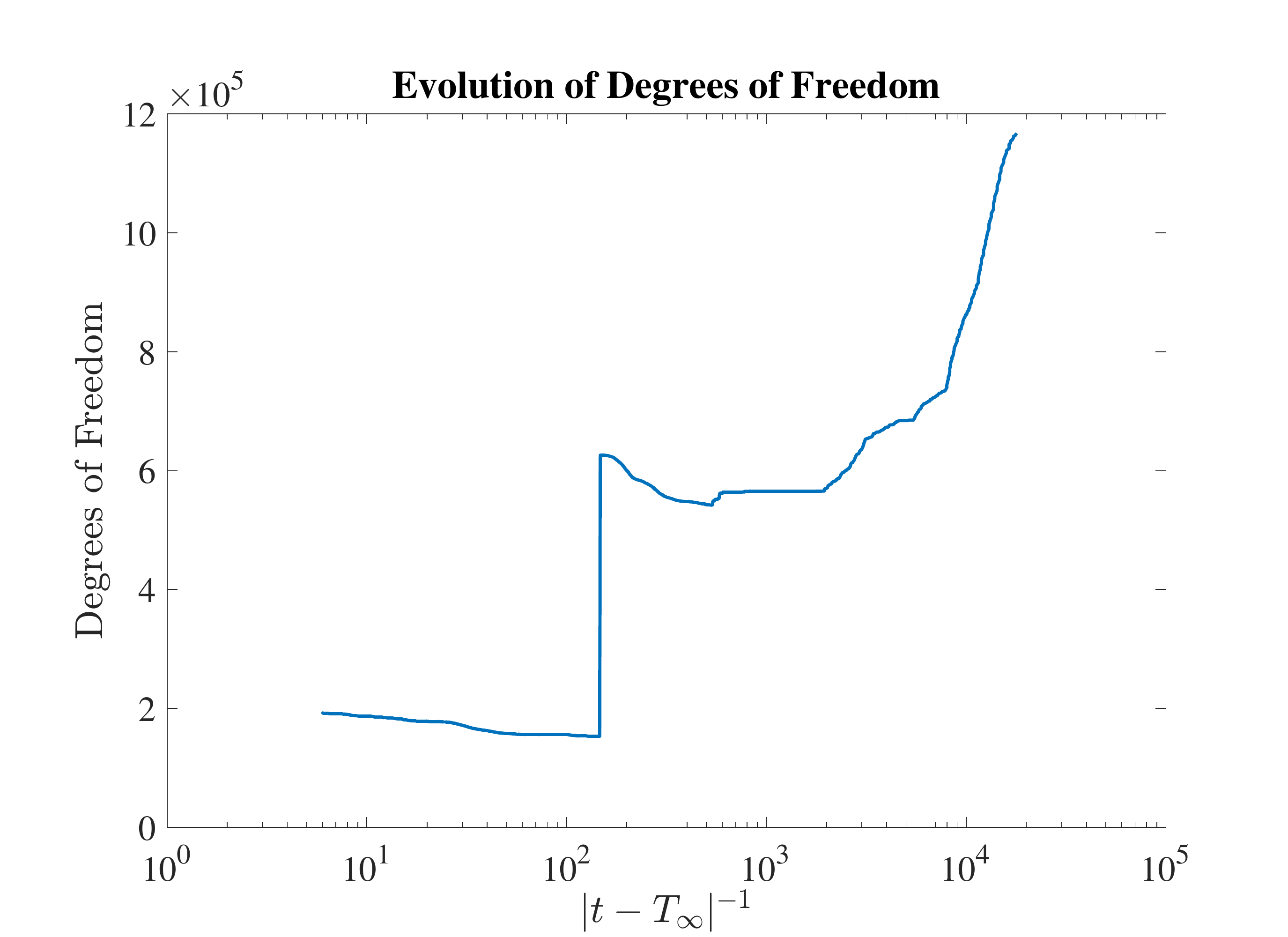}
\caption{\text{Example 2: convergence results (left) and evolution of degrees of freedom (right).}}
\label{example2results}
\end{figure}
\begin{figure}[h]
\centering
\includegraphics[scale=0.197]{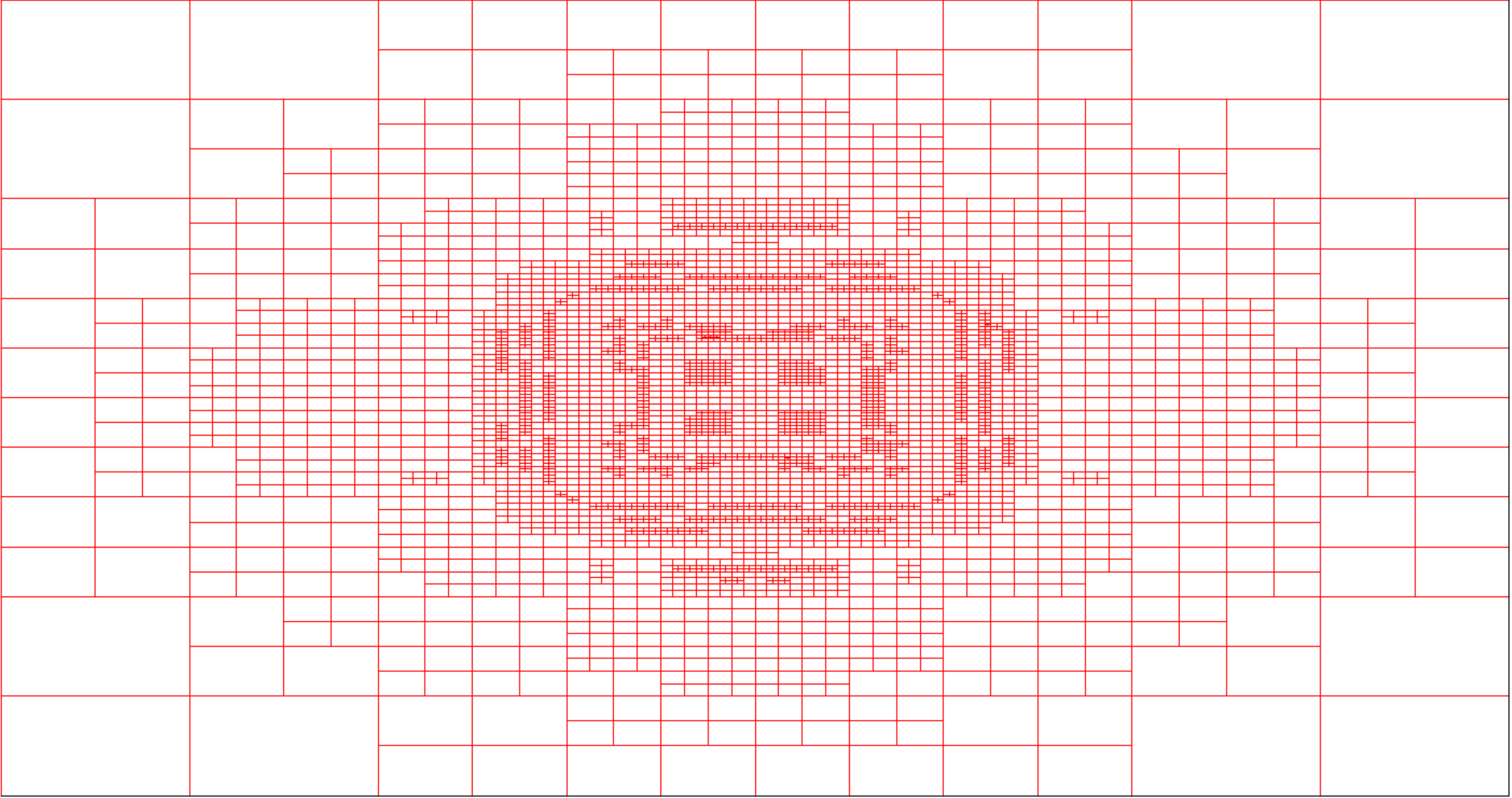} \hspace{1mm} \includegraphics[scale=0.197]{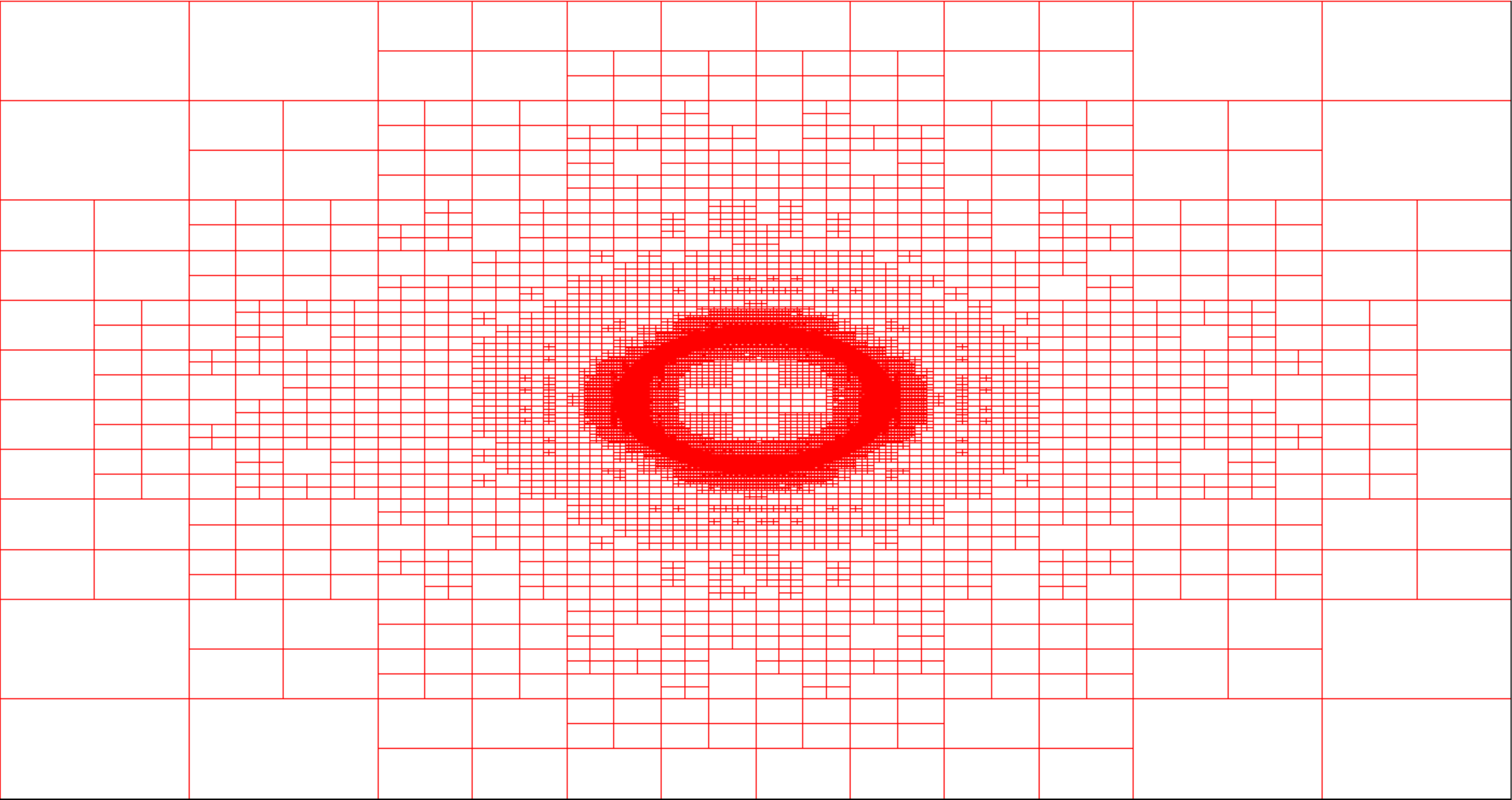}

\vspace{4mm}

\includegraphics[scale=0.184]{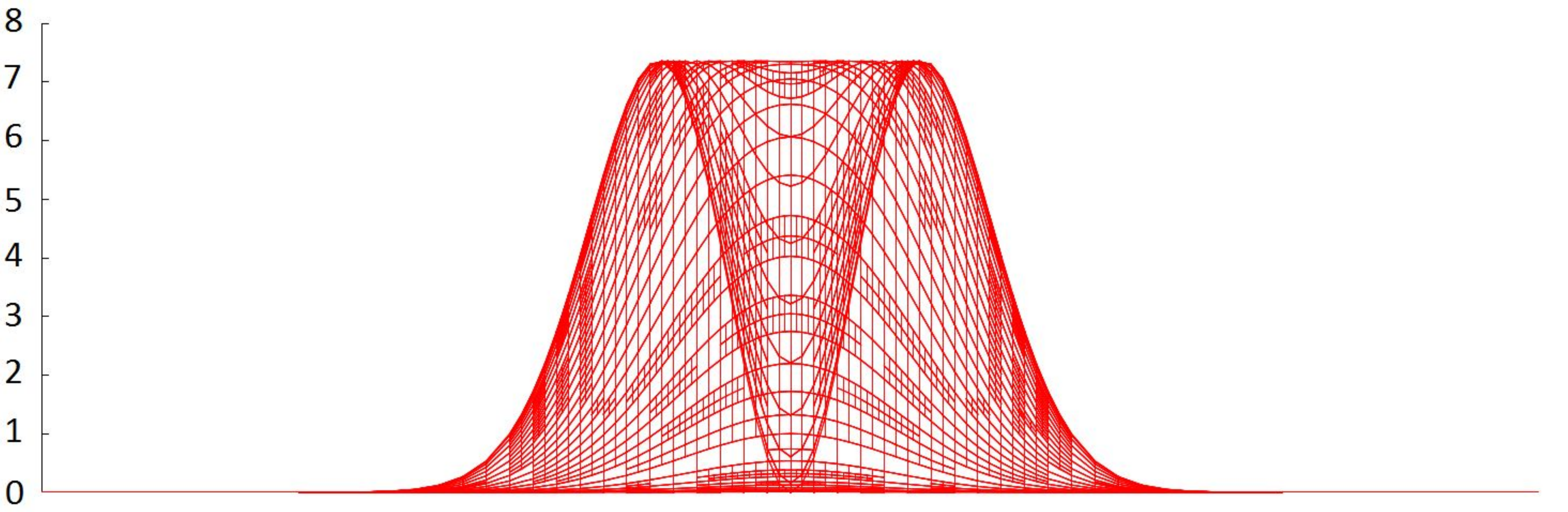} \hspace{-1.2mm}\includegraphics[scale=0.184]{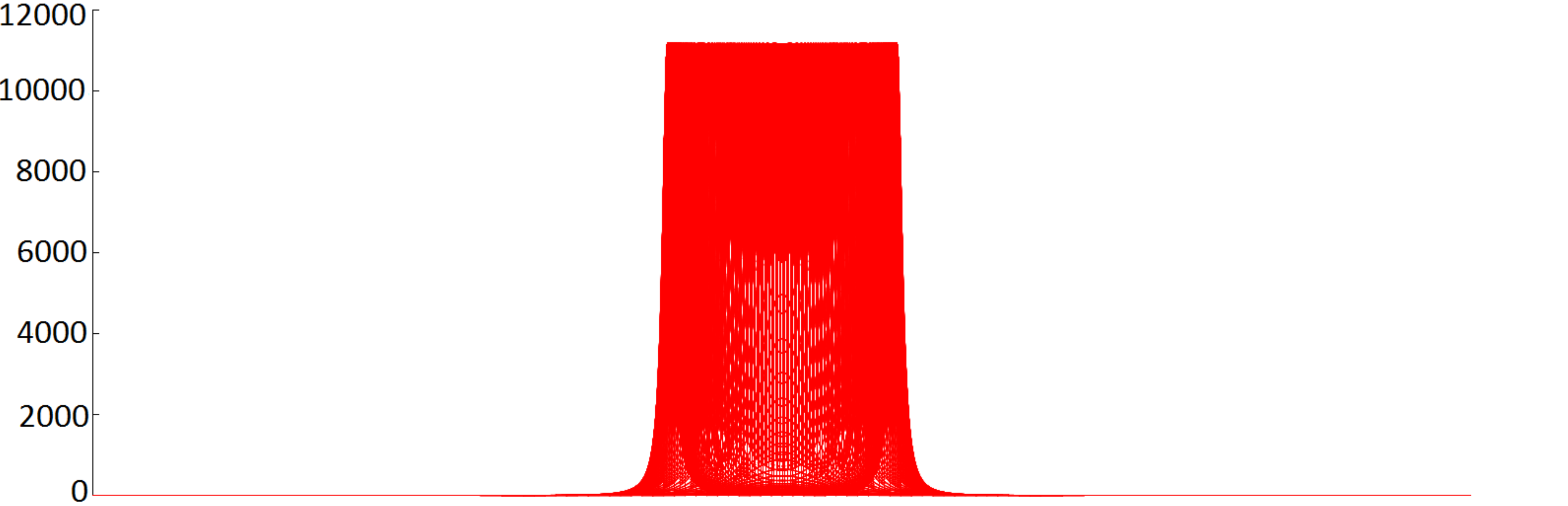}
\caption{Example 2: initial mesh (top left), final mesh (top right), initial solution profile (bottom left) and final solution profile (bottom right).}
\label{example2plots}
\end{figure}

As the nonlinearity is the same here as in Example 1, \eqref{blowuptimeapprox} is still valid and so we obtain the approximation $T_{\infty} \approx 0.166453$. Using this approximation to $T_{\infty}$, we take the data from Table \ref{blowupdata2} and plot the distance from the blow-up time $|T-T_{\infty}|$ versus the total number of time steps $N$ in Figure \ref{example2results}. The plot shows that for this example we have
\begin{equation}
\begin{aligned}
\notag
|T_{\infty} - T({\tt ttol^+}, N)| \propto N^{-1},
\end{aligned}
\end{equation}
which is what we expected to observe given the ODE results of \cite{CGKM16, KMW16, M15}. Next, we investigate the spatial properties of the adaptive algorithm during the final computational run. We begin by plotting the number of degrees of freedom versus the inverse of the distance to the blow-up time in Figure \ref{example2results}. The plot shows a general non-excessive increase in the number of degrees of freedom as we advance towards the blow-up time with some local decreases. To investigate this further, we display the meshes at times $t = 0$ and $t = T$ in Figure \ref{example2plots} which shows heavy refinement around the blow-up set and some derefinement in areas from which the finite element solution has retreated. For visualization purposes, we also display profile views of the finite element solution at times $t=0$ and $t=T$ in Figure \ref{example2plots}.

\subsection{Example 3} In this example, we consider a nonlinear parabolic problem from \cite{AW16, V05}. We set $\Omega = (0,1)^2$, $T = 0.75$, $a = 0.001$, $f(t,u) = \sin(t) - u^4$ and $u_0(x,y)=xy(x-1)(y-1)$. The solution is initially unremarkable but as time evolves it begins to exhibit boundary layers through the influence of the diffusion and the forcing term. For this nonlinearity, given any $t \in [0,T]$ and $v_1, v_2 \in \mathbb{R}$ we have
\begin{equation}
\begin{split}
|f(t,v_1)-f(t,v_2)| = |v_1^4 - v_2^4 | \leq |v_1 - v_2|(|v_1|^3+|v_1|^2|v_2| + |v_1||v_2|^2+|v_2|^3).
\end{split}
\end{equation}
Therefore, we have $\mathcal{L}(|v_1|,|v_2|) = |v_1|^3+|v_1|^2|v_2| + |v_1||v_2|^2+|v_2|^3$ in \eqref{eq:Lip} and so $\delta_m$ (if it exists) is the smallest root of the function $\phi_m:[1,\infty) \to \mathbb{R}$ given by
\begin{equation}
\begin{split}
\phi_m(\delta) = 1+ \delta\!\left[\int_{I_m} L(s, \delta) \dif s -1 \right]\! = 1 - \delta + 4\delta \int_{I_m} (\delta\psi_m + ||U(s)|| + C_{\infty}\xi_m)^3 \dif s,
\end{split}
\end{equation}
which we approximate via a Newton method.

Our primary goal in this numerical example is to verify that the estimator of Theorem \ref{maintheorem} is of optimal order in space and time when applied to a fixed-time problem. To that end, we begin by first checking the rate of convergence of the estimator in time. To do this, we choose a large polynomial degree and a small spatial refinement tolerance ${\tt stol}^+$ so that the spatial contribution to the estimator is negligible; we then gradually reduce the temporal refinement tolerance ${\tt ttol}^+$ in order to observe the rate of convergence of the estimator in time. Next, we plot the value of the estimator at final time versus the total number of time steps for the different computational runs in Figure \ref{example3results} {\bf --} the results show that the estimator is order one in time and, hence, optimal.

In order to analyze the rate of convergence of the estimator in space, we need to introduce a concept from \cite{CGM13} which is that of the \emph{weighted average degrees of freedom} given by
\begin{equation}
\begin{split}
\notag
\text{weighted average dofs} := \frac{1}{T} \sum_{m=1}^{M} k_m\lambda_m,
\end{split}
\end{equation}
where $\lambda_m$ is the number of degrees of freedom on the mesh $\mathcal{T}^{m-1} \vee \mathcal{T}^m$. In order to quantify the rate of convergence of the estimator in space, we first choose a small temporal refinement tolerance ${\tt ttol}^+$ so that the size of the temporal contribution to the estimator is negligible; we then gradually decrease the spatial refinement tolerance ${\tt stol}^+$ for polynomials of degree three and plot the value of the estimator at final time versus the weighted average degrees of freedom from the various computations in Figure \ref{example3results}. The results show that we obtain the expected, optimal rate of convergence in space. We also display meshes from one of the computational runs at times $t = 0$ and $t = T$ in Figure \ref{example3plots}. The initial mesh has some slight refinement around the boundary but is otherwise unremarkable whereas final mesh has significant refinement in the areas around the boundary suggesting that the spatial estimator has accurately captured the layers as they formed.

\begin{figure}[h]
\centering
\includegraphics[scale=0.265]{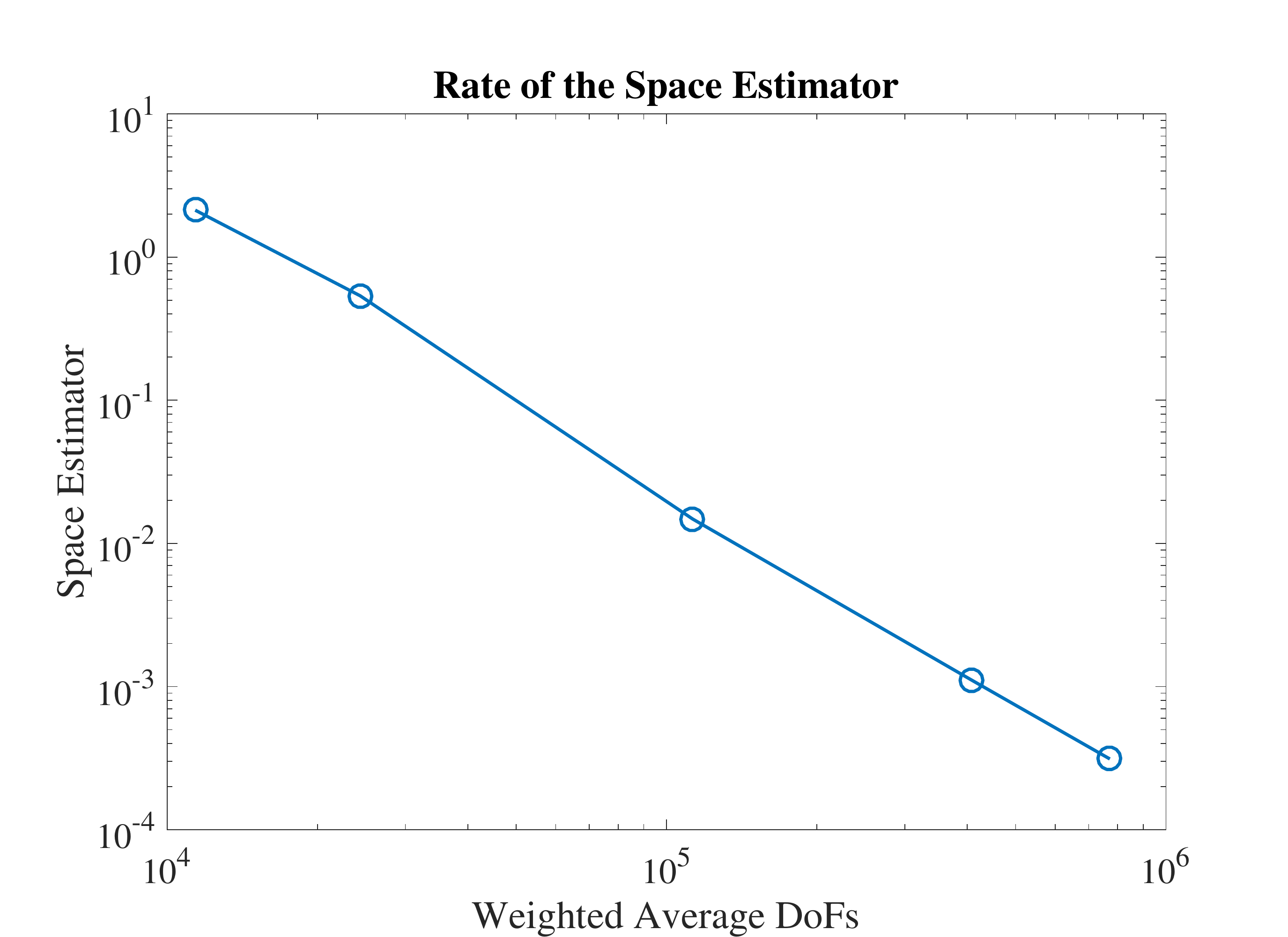} \includegraphics[scale=0.265]{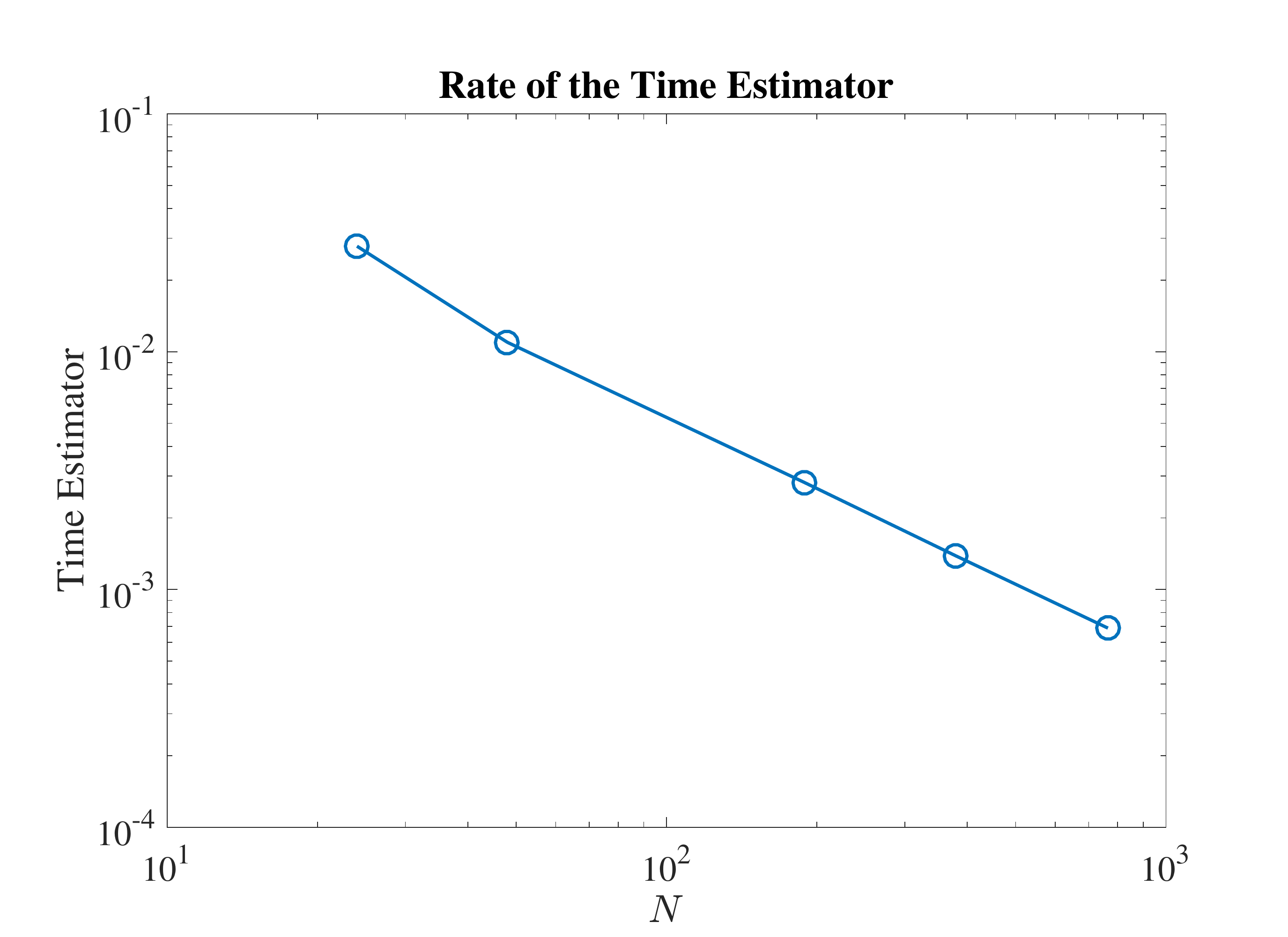}
\caption{\text{Example 3: spatial convergence rate for $p = 3$ (left) and temporal convergence rate (right).}}
\label{example3results}
\end{figure}

\begin{figure}[h]
\centering
\includegraphics[scale=0.196]{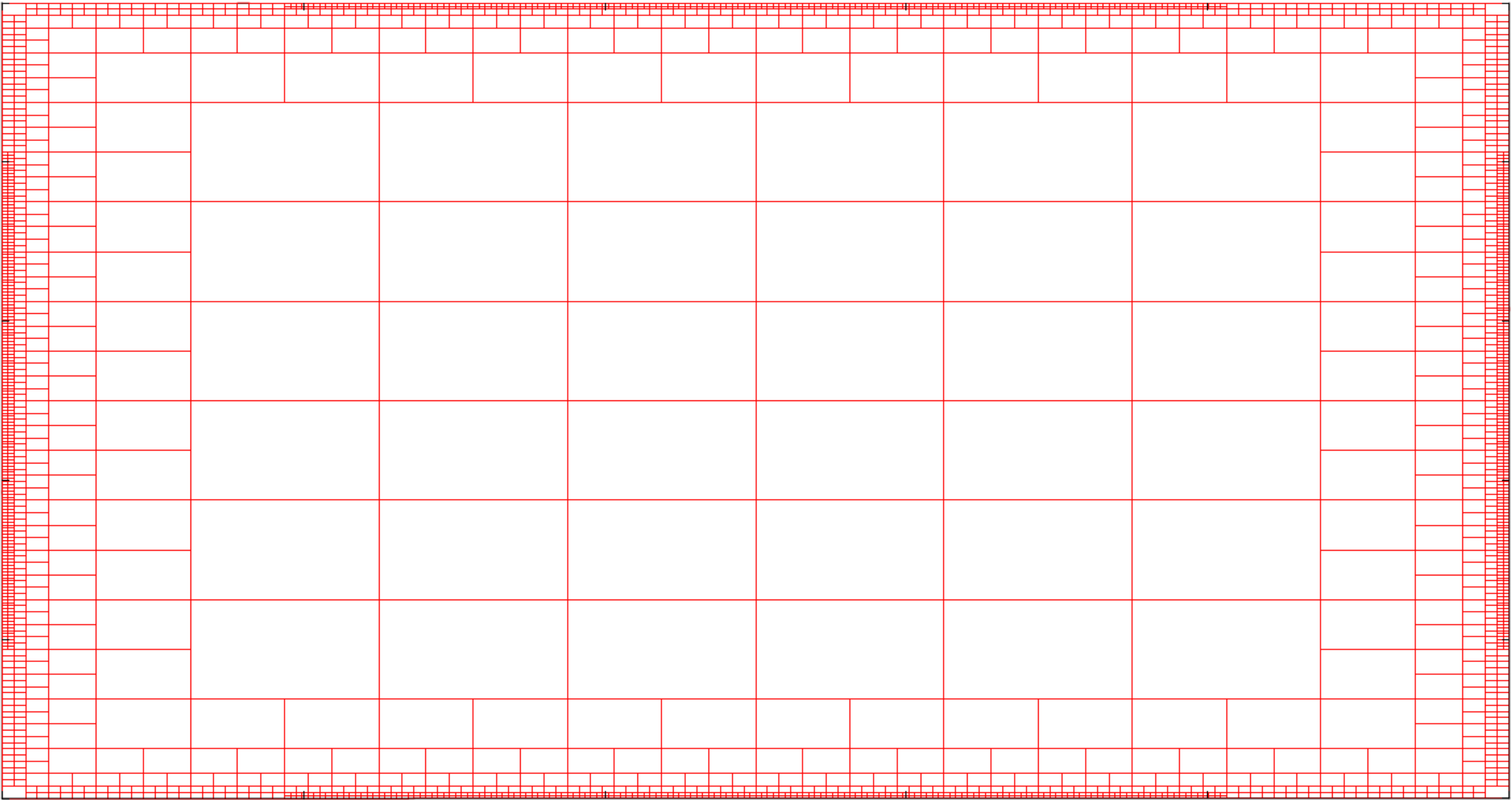} \hspace{1mm} \includegraphics[scale=0.196]{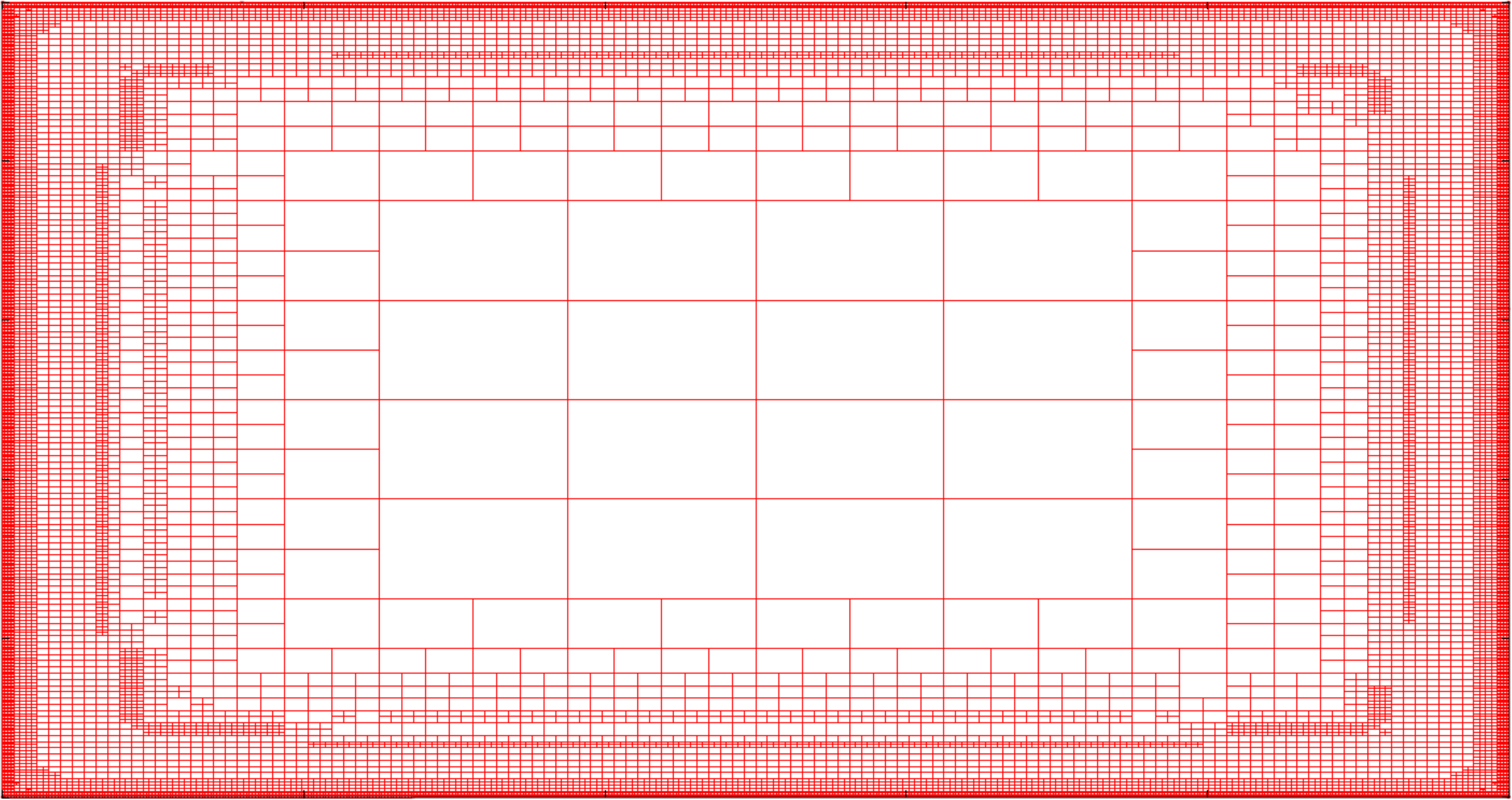}
\caption{Example 3: initial mesh (left) and final mesh (right).}
\label{example3plots}
\end{figure}

\section{Conclusions}

We derived a conditional $L^{\infty}L^{\infty}$ \emph{a posteriori} error bound (Theorem \ref{maintheorem}) for the IMEX discretization \eqref{IMEX} of the semilinear heat equation \eqref{model_strong} with general local Lipschitz nonlinearity \eqref{eq:Lip}. Our numerical experiments indicate that the proposed estimator outperforms the $L^2 H^1$ estimator of \cite{CGKM16} with respect to estimation of the blow-up time. Moreover, we were able to ascertain that the rate of convergence to the blow-up time is order one in the best-case scenario (Example 2) but we also determined that it can be slower (albeit still faster than in \cite{CGKM16}) in certain situations (Example 1). The slow convergence in Example 1 can be explained by the initial condition not having a ``compatible profile'' with the blow-up, that is, this choice of initial condition causes the solution to be significantly influenced by the laplacian early on (this can be seen in Figure \ref{blowupdata1} by noting that the numerical solution only achieves the estimate $||U(t)|| \propto |t - T_{\infty}|^{-1}$ asymptotically and not at all stages of the computation) in a way that is not suitably accounted for by the proposed estimator. Attempts were made to modify the estimator of Theorem \ref{maintheorem} to account for the influence of the laplacian in the spirit of Proposition 4.5 of \cite{DLM09} but this did not appear to have a significant impact on the performance of the estimator. We note, however, that a requirement on the initial condition to have a ``compatible profile'' with the nonlinearity in order to achieve optimal convergence is not unreasonable and has been a requirement in the \emph{a posteriori} error analysis of other nonlinear problems, for example, in \cite{KNS04}. Additionally, we verified in Example 3 that the estimator is of optimal order in space and time when applied to a fixed time problem. Indeed, we remark that condition \eqref{eq:Lip} is very general and that the estimator of Theorem \ref{maintheorem} can, in principle, be applied to \emph{any} nonlinear problem which satisfies it. In practise, however, the exponential term $\widetilde{r}_m$ restricts application of the estimator to nonlinear problems for which either

\smallskip

(1) the initial condition is small or,

\smallskip

(2) the nonlinearity is small or,

\smallskip

(3) the final time is small.

\smallskip

\noindent This also shows why the estimator works well for blow-up problems {\bf --} because a large nonlinearity corresponds to a small blow-up time ensuring that $\widetilde{r}_m$ never grows out of control. In the future, we would like to robustly incorporate the influence of the laplacian into the estimator of Theorem \ref{maintheorem}, prove convergence to the blow-up time under the proposed adaptive algorithm and explore the possibility of exponential convergence to the blow-up time in the spirit of \cite{KMW16}.

\section{Acknowledgements}

The research in this paper was conducted while the second author was affiliated with Universit\"at Bern and he is thankful to them for supporting the research. The authors would also like to thank Prof. Thomas P. Wihler of Universit\"at Bern for his comments and insights.

\bibliographystyle{amsplain}
\bibliography{paper1}
\end{document}